%% file: starfleet_cmame.tex
\documentclass[12pt]{article}

\usepackage{amssymb}
\usepackage{amsthm}
\theoremstyle{plain}
\usepackage{afterpage,capt-of}
\usepackage{natbib}
\usepackage[a4paper,top=3cm,bottom=3cm,left=3cm,right=3cm]{geometry}

\newtheorem{prop}{Property}
\theoremstyle{remark}
\newtheorem{remark}{Remark}

\usepackage{comment}
\usepackage[T1]{fontenc}
\usepackage[utf8]{inputenc}
\usepackage{graphicx}
\usepackage{amsmath,amsfonts,amssymb,amsbsy,stmaryrd,amsthm}
\usepackage{tikz}
\usepackage{pgfplots}
\SetSymbolFont{stmry}{bold}{U}{stmry}{m}{n}
\makeatletter
\newif\if@restonecol
\makeatother

\usepackage[algo2e,ruled,lined,titlenumbered,commentsnumbered]{algorithm2e}
\usepackage{lscape}
\usepackage{xcolor}

\newcommand{\un}[1]{\ensuremath{\underline{#1}}}
\newcommand{\uu}[1]{\ensuremath{\un{\un{#1}}}}
\newcommand{\dime}{d}

\newcommand{\sig}{\ensuremath{\uu{\sigma}}}
\newcommand{\sigH}{\ensuremath{\sig_H}}
\newcommand{\sigp}{\ensuremath{\sig_p}}
\newcommand\strain[1]{\uu{\varepsilon}\left(#1\right)}
\newcommand{\dep}{\ensuremath{\un{u}}}
\newcommand{\depv}{\ensuremath{\un{v}}}
\newcommand{\depH}{\ensuremath{\dep_H}}
\newcommand{\depp}{\ensuremath{\dep_p}}
\newcommand{\depvH}{\ensuremath{\depv_H}}

\newcommand{\depN}{\ensuremath{\mathbf{u}}}
\newcommand{\lamN}{\ensuremath{\boldsymbol{\lambda}}}
\newcommand\shapef{\varphi}
\newcommand\shapev{\boldsymbol{\shapef}_H}
\newcommand{\stiff}{\ensuremath{\mathbf{K}}}
\newcommand{\force}{\ensuremath{\mathbf{f}}}
\newcommand{\domain}{\ensuremath{\Omega}}
\newcommand{\domainH}{\ensuremath{\domain}}
\newcommand{\setelem}{\ensuremath{\mathcal{T}}}
\newcommand{\setedge}{\ensuremath{\mathcal{E}}}

\newcommand{\setface}{\ensuremath{\mathcal{F}}}
\newcommand{\setvertex}{\ensuremath{\mathcal{V}}}

\newcommand{\setedgei}{\ensuremath{\overset{\circ}{\mathcal{E}}}}
\newcommand{\setedgeii}{\ensuremath{\overset{\circ\circ}{\mathcal{E}}}}

\newcommand{\setfacei}{\ensuremath{\overset{\circ}{\mathcal{F}}}}
\newcommand{\setvertexi}{\ensuremath{\overset{\circ}{\mathcal{V}}}}

\newcommand{\setedgee}{\ensuremath{\mathcal{E}_\partial}}

\newcommand{\setvertexe}{\ensuremath{{\mathcal{V}_\partial}}}
\newcommand{\setfacee}{\ensuremath{\mathcal{F}_\partial}}

\newcommand{\card}[1]{\ensuremath{|#1|}}

\newcommand{\bound}{\ensuremath{\partial}}
\newcommand{\boundi}{\ensuremath{\tilde{\partial}^{-1}}}
\newcommand{\boundii}{\ensuremath{\tilde{\partial}^{-2}}}

\newcommand\hooke{\mathbb{H}}
\newcommand\KA{\ensuremath{\mathrm{KA}}}
\newcommand\KAo{\ensuremath{\KA^0}}
\newcommand\KAH{\ensuremath{\KA_H}}
\newcommand\KAp{\ensuremath{\KA_p}}
\newcommand\KAHo{\ensuremath{\KA_H^0}}
\newcommand\SA{\ensuremath{\mathrm{SA}}}

\newcommand\RBM{\ensuremath{\mathrm{RBM}}}

\newcommand\ecr[2]{\ensuremath{\mathrm{e_{CR_{#2}}(#1)}}}
\newcommand\enernorm[2]{\|#1\|_{\hooke^{-1},#2}}

\newcommand{\nga}{\ensuremath{\un{n}^\Gamma}}
\newcommand{\tu}{\ensuremath{\hat{\dep}}}
\newcommand{\ts}{\ensuremath{\hat{\sig}}}
\newcommand{\tF}{\ensuremath{\hat{\un{F}}^\Gamma}}
\newcommand{\stF}{\ensuremath{\hat{{F}}^\Gamma}}
\newcommand{\sFH}{\ensuremath{{F}_H^\Gamma}}
\newcommand{\FH}{\ensuremath{\un{F}_H^\Gamma}}

\newcommand{\matD}{\ensuremath{\boldsymbol{\Delta}}}

\newcommand{\veci}{\ensuremath{{\mathbf{i}}}}
\newcommand{\vecR}{\ensuremath{{\mathbf{R}}}}

\newcommand{\vecW}{\ensuremath{\hat{\mathbf{W}}}}
\newcommand{\vecWnh}{\ensuremath{{\mathbf{W}}}}
\newcommand{\matG}{\ensuremath{\mathbf{G}}}
\newcommand{\mata}{\ensuremath{\mathbf{a}}}
\newcommand{\matb}{\ensuremath{\mathbf{b}}}
\newcommand{\matc}{\ensuremath{\mathbf{c}}}
\newcommand{\matd}{\ensuremath{\mathbf{d}}}
\newcommand{\mate}{\ensuremath{\mathbf{e}}}

\newcommand{\matl}{\ensuremath{\mathbf{l}}}
\newcommand{\maty}{\ensuremath{\mathbf{y}}}
\newcommand{\matx}{\ensuremath{\mathbf{x}}}

\newcommand{\matU}{\ensuremath{{\mathbf{U}}}}
\newcommand{\matV}{\ensuremath{{\mathbf{V}}}}
\newcommand{\matB}{\ensuremath{\mathbf{B}}}

\newcommand{\matI}{\ensuremath{{\mathbf{I}}}}

\newcommand{\matO}{\ensuremath{\mathbf{0}}}
\newcommand{\matN}{\ensuremath{\mathbf{N}}}
\newcommand{\kerG}{\ensuremath{\mathbf{Z}}}
\newcommand{\matM}{\ensuremath{\mathbf{M}}}
\newcommand{\vvecW}{\ensuremath{\un{\vecW}}}
\newcommand{\tvecR}{\ensuremath{{\tilde{\vecR}}}}
\newcommand{\basis}{\ensuremath{\mathcal{B}}}
\newcommand{\meas}{\ensuremath{\operatorname{meas}}}
\newcommand{\R}{\ensuremath{\mathbb{R}}}

\title{Study of the strong prolongation equation for the construction of statically admissible stress fields: implementation and optimization}
\author{Valentine Rey, Pierre Gosselet, Christian Rey\footnote{valentine.rey@lmt.ens-cachan.fr,gosselet@lmt.ens-cachan.fr,christian.rey@lmt.ens-cachan.fr} \\
LMT-Cachan / ENS-Cachan, CNRS, UPMC, Pres UniverSud Paris \\
61, avenue du pr\'esident Wilson, 94235 Cachan, France \\
}
\begin{document}




\maketitle
\begin{abstract}
This paper focuses on the construction of statically admissible stress fields (SA-fields) for \textit{a posteriori} error estimation. In the context of verification, the recovery of such fields enables to provide strict upper bounds of the energy norm of the discretization error between the known finite element solution and the unavailable exact solution. The reconstruction is a difficult and decisive step insofar as the effectiveness of the estimator strongly depends on the quality of the SA-fields. This paper examines the strong prolongation hypothesis, which is the starting point of the Element Equilibration Technique (EET). We manage to characterize the full space of SA-fields satisfying the prolongation equation so that optimizing inside this space is possible. 
The computation exploits topological properties of the mesh so that implementation is easy and costs remain controlled. In this paper, we describe the new technique in details and compare it to the classical EET and to the flux-free technique for different 2D mechanical problems.
The method is explained on first degree triangular elements, but we show how extensions to different elements and to 3D are straightforward.


\end{abstract}

\section{Introduction}
\input{introduction}

\section{Problem setting and error estimation}\label{sec:setting}
\input{setting}
\section{Classical EET technique}\label{sec:EET}
\input{classical_eet}
\section{New SA-field reconstruction technique}\label{sec:STARFLEET}
\input{newsa}

\section{Solution strategy}\label{sec:solving}
\input{solving}
\section{Assessment}\label{sec:assessment}
\input{assessment}
\section{Extensions}\label{sec:extensions}
\input{extensions}

\section{Conclusion}
\input{conclusion.tex}

\paragraph*{Acknowledgments}
The authours would like to thank Professor Pedro Diez (Universitat Polit\`{e}cnica de Catalunya) and Florent Pled (LMT-Cachan) for the helpful and fruitful discussions, and Eric Florentin (LMT-Cachan) for the help in the estimation based on a dual approach.

\bibliographystyle[authoryear]{natbib}

\bibliography{Biblio}







\end{document}

%% file: introduction.tex
This paper deals with the post-processing of statically admissible (SA) stress fields from a classical finite element (FE) solution in displacement. This step is a requirement of verification techniques if one wants precise, strict and constant-free bounds between the FE solution and the exact (unavailable) solution.

Indeed methods based on the use of SA-field, like the error in constitutive equation \citep{Lad75,Lad83} or like the equilibrated residuals \citep{Bab78b,Bab87}, provide efficient error estimators, for global or local quantities \citep{Lad99b,Pru99,Ohn01,Bec01}, even in some nonlinear contexts \citep{Bab82,Lad86,Lad01bis,Lou03bis} and domain decomposition methods \citep{Par10}. 

The computation of SA-fields is a complex task and depending on the quality of the field, the error estimators provide an effective bound or not. With respect to that criterion, the optimal SA-field is the one which minimizes the complementary energy. In order to build such a field, a possibility is to use a dual analysis \citep{Kem03,Kem09}. However dual codes are not very common and the computational cost for such a computation are very high.

An alternative approach consists in post-processing the SA-field from the FE solution. The Element Equilibration Technique \citep{Lad83,Lad97}, its variants \citep{Lad10bis,Ple11}, or the Flux-Free techniques \citep{Par06,Par09bis,Cot09} are typical examples of such alternatives. Compared to the dual approach, they involve many short-range computations (on star-patches and elements), they are naturally pluggable inside classical FE software, but they give less efficient estimates. Even if the computational complexity is much reduced compared to the dual approach, they remain complex to implement and CPU-time consuming.

This paper revisits the Element Equilibration Technique, and more precisely one of its key ingredient: the strong prolongation equation. This equation guides the construction of the SA-field by imposing that it develops the same virtual work as the FE solution within the FE displacement fields. By a global analysis, we manage to obtain, at a low cost, the full set of SA-fields which satisfy the strong prolongation equation.

This new way of interpreting the strong prolongation equation corresponds to forming a large sparse linear system. A first advantage is that this setting clearly separates the different roles star-patches can play: the left kernel corresponds to star-patches seen as support of shape functions, the right kernel corresponds to star-patches seen as closed fluxes. Moreover, the implementation is straightforward;  most operations are naturally vectorized and applied on blocks of data, and fast solvers can be employed because well-posed sparse problems can always be easily extracted. Various optimization criteria can be tested, which enables us to evaluate the pertinence of the strong prolongation equation. In particular we obtain an optimal estimator (but at prohibitive computational costs), and an estimator equivalent to the classical EET in term of quality but faster in term of CPU time.
%
The acronym of the method is \textsc{starfleet}, for star-free and lazy element equilibration technique, where we insist on the ease of implementation and on the fact that star-patches are not an ingredient but a natural feature of the method.

In order to ease the presentation, the method is developed for 2D linear elasticity problems approximated by first order triangular elements, for which numerical assessments are proposed. The extensions of the method are discussed at the end of the article. The rest of the paper is organized as follow: in Section~\ref{sec:setting} we define the problem and the notations, in Section~\ref{sec:EET} we recall the original EET technique, then we propose in Section~\ref{sec:STARFLEET} an alternative implementation of the strong prolongation equation which generalizes the classical EET and offers new  optimization possibilities. In order not to interrupt the presentation, the technical details on which the method relies are exposed in Section~\ref{sec:solving}. 
Assessments are given in Section~\ref{sec:assessment}. In Section~\ref{sec:extensions}, we show that the method is not limited to first degree triangles but can be straightforwardly extended to second degree triangular elements, quadrangular elements and 3D domains.

%% file: setting.tex
Let $\mathbb{R}^\dime$ ($\dime=2$) represents the physical space. 
Let us consider the static equilibrium  of a (polygonal) structure which occupies the open domain $\domain\subset\mathbb{R}^\dime$ and which is subjected to  given body force $\un{f}$ within $\Omega$,  to given
traction force~$\un{g}$ on~$\partial_g\Omega$  and to given displacement field~$\dep_d$ on the  complementary part  $\partial_u\Omega\neq\emptyset$. We  assume that the structure undergoes  small   perturbations  and  that  the  linear elastic material   is  characterized by Hooke's  elasticity tensor~$\hooke$.  Let~$\dep$ be  the unknown displacement field, $\strain{\dep}$ the symmetric part  of the gradient of~$\dep$, $\sig$ the Cauchy stress tensor.

We introduce two affine subspaces and one positive form:
\begin{itemize}
\item Affine subspace of kinematically admissible fields (KA-fields)
\begin{equation}\label{eq:KA}
  \KA=\left\{ \dep\in \left(\mathrm{H}^1(\domain)\right)^\dime,\ \dep = \dep_d \text{ on }\partial_u\domain \right\}
\end{equation}
and we note $\KAo$ the associated vectorial space.
\item Affine subspace of statically admissible fields (SA-fields)
\begin{multline}\label{eq:SA}
  \SA
  =\Bigg\lbrace   \uu{\tau}\in  \left(\mathrm{L}^2(\domain)\right)^{\dime\times \dime}_{\text{sym}}; \
    \forall  \depv \in  \KAo,\ \\ \int\limits_\domain
  \uu{\tau}:\strain{\depv} \,   d\domain    =    \int\limits_\domain   \un{f} \cdot\depv   \, d\domain +
  \int\limits_{\partial_g\domain} \un{g}\cdot\depv\, dS   \Bigg\rbrace
\end{multline}
\item Error in constitutive equation 
\begin{equation}\label{eq:ecr}
  \ecr{\dep,\sig}{\domain}= \enernorm{\sig-\hooke:\strain{\dep}}{\domain}
\end{equation}
where ${\enernorm{\uu{x}}{\domain}}=\displaystyle \sqrt{\int_\domain \left( \uu{x}: {\hooke}^{-1}:\uu{x} \right)d\domain}$
\end{itemize}

The mechanical problem set on $\Omega$ can be formulated as:
\begin{center}
  Find    $\left(\dep_{ex},\sig_{ex}\right)\in\KA\times\SA$    such    that
  $\ecr{\dep_{ex},\sig_{ex}}{\domain}=0$
\end{center}
The solution to this problem, named ``exact'' solution, exists and is unique.

\subsection{Finite element approximation}
{We consider the discretization of the geometry by triangles (meshing of the domain). To that subdivision},  we  associate a finite-dimensional  subspace  $\KAH$  of $\KA$.  The  classical finite element displacement approximation consists in searching:
\begin{equation}\label{eq:feeq}
  \begin{aligned}
    \depH&\in\KAH\\
    \sigH&=\hooke:\strain{\depH}   \\
    \int_{\domainH}
  \sigH:\strain{\depvH}   d\domain  &=   \int_{\domainH} \un{f}\cdot\depvH   d\domain   +
  \int_{\partial_g\domainH} \un{g}\cdot\depvH dS,{\ \forall \depvH\in\KAHo}
    \end{aligned}
\end{equation}
Of course the approximation is due to the fact that in most cases $\sigH\notin\SA$.

After introducing the matrix $\shapev$ of  shape functions which form a basis of $\KAH$  and the  vector of  nodal unknowns  $\depN$ so that $\depH=\shapev \depN$, the  classical finite  element method leads  to the linear
system:
\begin{equation}\label{eq:globalFE}
\stiff \depN = \force
\end{equation}
where $\stiff$ is  the (symmetric positive definite) stiffness  matrix and $\force$ is the vector of generalized forces. Classically, the finite element space is extended to degrees of freedom where Dirichlet boundary conditions are imposed, and the vector of nodal reactions $\lamN_d$, which correspond to the work of the finite element stress field in the shape functions, can be deduced:
\begin{equation}
\lamN_d= \left(\int_{\partial_u\domain} (\sigH \cdot \underline{n}) \cdot {\shapev}_d\,dS\right)^T = \stiff_{dr} \depN + \stiff_{dd}\depN_d -\force_d
\end{equation}
where Subscript $d$ stands for Dirichlet degrees of freedom (so that $\depN_d$ is known) and  Subscript $r$ corresponds to the remaining degrees of freedom which were determined in \eqref{eq:globalFE}, $\underline{n}$ is the normal vector to $\partial_u\domain$. 

\subsection{Principle of the error in constitutive equation}
The measurement of the exactness of the solution through the error in constitutive equation consists in deducing admissible fields from the computed FE fields and measuring the non-verification of the constitutive equation:
\begin{equation}
\left\{\begin{array}{l}
\depH \in \KAH\\
\sigH =\hooke:\strain{\depH} \\
\ecr{\depH,\sigH}{\domainH}=0
\end{array}\right.
\longrightarrow
\left\{\begin{array}{l}
\tu \in \KA\\
\ts \in \SA\\
\ecr{\tu,\ts}{\domainH}\geqslant 0
\end{array}\right.
\end{equation}
$\ecr{\tu,\ts}{\domainH}$ gives strict bounds of the strain and stress errors in the energy norm \citep{Lad75,Lad04}.

In the case we consider, the construction of $\tu\in\KAH$ is simple because $\KAH\subset\KA$ so that we choose $\tu=\depH$.  Building $\ts$ is a more complex task, for which many techniques have been proposed, among others the element equilibration technique  \citep{Lad83,Lad97} (EET), the flux-free technique \citep{Par06} and the star-patch element equilibration technique \citep{Lad10bis,Ple12}.
This paper revisits the EET technique, about which we give more details in Section~\ref{sec:EET}. 

Once $\ts\in\SA$ is built, the computation of the error is performed over each element $T$ as follow:
\begin{equation}\label{eq:computeerror1}
\begin{aligned}
\ecr{\depH,\ts}{\domainH} &= \sum_T \ecr{\depH,\ts}{T} \\
\ecr{\depH,\ts}{T}&=\frac{1}{2}\int\limits_T (\ts-\hooke:\strain{\depH}):\hooke^{-1}:(\ts-\hooke:\strain{\depH})d\Omega\\
&= \frac{1}{2}\int\limits_T  \ts:\hooke^{-1}:\ts d\Omega  + \frac{1}{2}\int\limits_T  \strain{\depH}:\hooke:\strain{\depH} d\Omega \\&\qquad\qquad\qquad\qquad\qquad  - \int\limits_T \ts:\strain{\depH} d\Omega
\end{aligned}
\end{equation}
where we recognize the deformation energy of the SA-field, the deformation energy of the KA-field and a coupling between the SA-field and the KA-field.

\subsection{Notations for the handling of the mesh}

We note $\setelem$  the set of elements (triangles) in the subdivision of $\domainH$, $\setedge$ the set of edges, and $\setvertex$ the set of vertexes. 
We note $\card{.} $ the cardinality of  sets (counting measure), so that for instance $\card{\setelem}$ is the number of elements in the mesh. 

We introduce the boundary operator $\bound$ which extracts the border edges of a group of elements and the ending vertexes of a set of edges (notation $\mathcal{P}(X)$ stands for the set of subsets of $X$). 
\begin{equation}
\mathcal{P}\left(\setelem\right) \overset{\bound}{\longrightarrow} \mathcal{P}\left(\setedge\right) \overset{\bound}{\longrightarrow}  \mathcal{P}\left(\setvertex\right)
\end{equation}

{We use the notation $\boundi$ for the mapping which gives the edges which are ended by one vertex and the elements which share one edge (when applied to sets,  $\boundi$ gives the union of the images of each term).  
In particular, for a vertex $V$ the elements in $\boundii V$ form the so called star-patch (set of elements having $V$ as a vertex). 

Figure~\ref{fig:partial} illustrates the boundary operator on simple cases.

\begin{figure}[ht]\centering
\includegraphics[width=.8\textwidth]{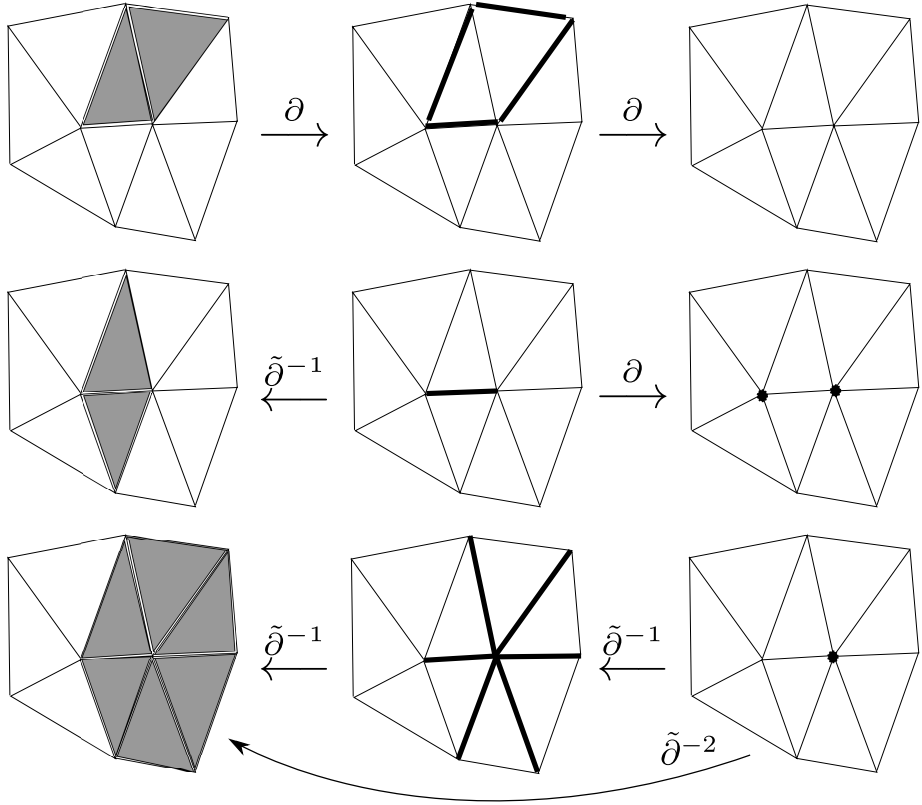}\caption{Illustration of the boundary operators}\label{fig:partial}
\end{figure}

%% file: classical_eet.tex
When building statically admissible stress fields, two conditions have to be met: the continuity of the fluxes between subregions and the verification of the local balance equation. In the Flux- Free technique, regularity is ensured by the use of the partition of unity and balance is satisfied through independent computations on star-patches. In the Element Equilibration Technique, the balance is verified by element-wise computations (leading to a family of element stresses $(\ts_T)$) whereas the continuity is ensured by the introduction of specific force unknowns on the edges between elements.

An extra criterion is introduced in order to limit the search space of statically admissible stress fields and simplify the determination of $(\ts_T)$, the so called strong-prolongation hypothesis:
\begin{equation}
\forall T \in \setelem,\ \forall \depv \in\KAH,\ \int_T (\ts-\sigH):\strain{\depv }d\Omega =0
\end{equation}
which means that the reconstructed stress field should develop the same amount of work than the original FE stress field in any FE deformation field. 

Moreover, this hypothesis enables to remove the coupling term in \eqref{eq:computeerror1}: 
\begin{equation}\label{eq:computeerror2}
 \begin{aligned}
&\int\limits_T\! \ts:\strain{\depH} d\Omega = \int\limits_T \sigH\!:\strain{\depH} d\Omega =\int\limits_T  \strain{\depH}:\hooke:\strain{\depH} d\Omega \\
&\Rightarrow \ecr{\depH,\ts}{T}= \frac{1}{2}\int\limits_T  \ts:\hooke^{-1}:\ts d\Omega  -\frac{1}{2}\int\limits_T  \strain{\depH}:\hooke:\strain{\depH} d\Omega
\end{aligned}
\end{equation}
Consequently, only the strain energy of the SA-field is required, and of course the smaller this energy the sharper the estimation. By decoupling the quantities from the original problem ($\depH$) and the dual problem ($\ts$), this expression avoids having to evaluate the primal field in the dual discretization.

The EET technique is a two-step procedure: first {(subsections \ref{ssec:introF}, \ref{ssec:classicF})} balanced tractions $\tF$ are constructed on the edges $(\Gamma)$ of elements, second {(subsection \ref{ssec:classicEc})} the SA-stress fields are computed independently on each element $(T)$ using the tractions as Neumann boundary conditions $\ts_T\cdot {\nga_T} =\delta_T^\Gamma\tF$ where ${\nga_T}$  is the  normal vector to face $\Gamma$ pointing outward from element $T$ and the role of $\delta_T^\Gamma=\pm 1$ is explained in Equation~\eqref{eq:delta}.

\subsection{Introduction of the balanced traction forces}\label{ssec:introF}

On each edge $\Gamma$, we wish to build a traction field $\tF$ which gives correct Neumann conditions for the computation of the statically admissible field $\ts$ which satisfies the strong prolongation condition.

\begin{equation}\label{eq:introF}\left\{
\begin{aligned}
&\ts_T\cdot{\nga_T}=\delta_T^\Gamma\tF,\ \forall T\in\setelem,\ \forall \Gamma \in \bound T  \\
&\tF = \un{g} \text{ on }\partial_g\domainH  \\
&0 =\int_T (\ts-\sigH):\strain{\depv }d\Omega  ,\ \forall \depv \in\KAH,
\end{aligned}\right.
\end{equation}
$\delta_T^\Gamma=\pm 1$ is used to ensure the action-reaction principle between two neighboring elements:
\begin{equation}\label{eq:delta}\left\{
\begin{aligned}
&\forall \Gamma \in  \partial\domainH,\ T=\boundi \Gamma,\ \delta_T^\Gamma = 1 \\
&\forall \Gamma \in \setedge\setminus \partial\domainH,\ \{T,T'\}=\boundi \Gamma,\ \delta_T^\Gamma + \delta_{T'}^\Gamma= 0 
\end{aligned}\right.
\end{equation}

An important property is that, because rigid body motions belong to FE field $\KAH$, the strong prolongation equation implies the equilibrium of elements with respect to rigid body motions: 
\begin{equation}
\forall T\in\setelem, \forall \un{\rho} \in\RBM,\ \int_T \un{f}\cdot\un{\rho} d\domain +  \sum_{\Gamma\in\bound T}\int_{\Gamma}\delta_T^\Gamma \tF\cdot\un{\rho}\, dS = 0
\end{equation}
where $\RBM$ is the set of translations and infinitesimal rotations. Then the strong prolongation condition ensures the verification of Fredholm alternative and $\ts$ is well defined {independently on elements}.

{If we develop the strong prolongation equation using the definition of tractions  \eqref{eq:introF} and the FE equilibrium \eqref{eq:feeq}, we obtain:}
\begin{equation}\label{eq:prolo_resid}
\begin{aligned}
&\forall T\in\setelem,{\ \forall \depv \in\KAH,}
\ \sum_{\Gamma\in\bound T} \delta_T^\Gamma \int_\Gamma \tF\cdot \depv  \,dS =\overset{\circ}{R}_T(\depv ) \\
&\text{where } \overset{\circ}{R}_T(\depv ) :{=} \int_T ( \sigH:\strain{\depv }-\underline{f}\cdot\depv ) d\Omega
\end{aligned}
\end{equation}
the internal residual $\overset{\circ}{R}_T(\depv )$ is then to be computed for each shape function on each element.

\subsection{Classical computation of the balanced edge tractions}\label{ssec:classicF}
In the original EET technique \citep{Lad83}, edge tractions $\tF$ are assumed to vary linearly, so that two coefficients per component are to be determined per edge. 

To determine these coefficients, the residual equation \eqref{eq:prolo_resid} is tested against shape functions so that computations are limited to the support of shape functions, called star-patches. The support of the shape function $\phi_H^N$  associated to the vertex $N\in\setvertex$ coincides with the star patch $\boundii N$ (see Figure~\ref{fig:partial}). 
We can compute the vector $\un{\hat{W}}_N^\Gamma$ corresponding to the work of edge tractions in the shape function (in each direction):
\begin{equation}
\un{\hat{W}}_N^\Gamma=\int_\Gamma \begin{pmatrix}\tF\cdot \un{e}_x {\phi}_H^N \\\tF\cdot \un{e}_y {\phi}_H^N \end{pmatrix}\,dS
\end{equation}
The strong prolongation equation can be rewritten as a set of decoupled small linear systems for each vertex of the mesh, the number of unknowns being the number of edges connected to that vertex. We have the following equation for each element of the star-patch:
\begin{equation}\label{eq:sysEET_SP}
{\text{for a given vertex }N}\ \left\{\begin{aligned}
&\forall T \in \boundii N,\text{ let } \bound T \cap \boundi N = \{\Gamma_1,\Gamma_2\},\\&  \delta_T^{\Gamma_1}\un{\hat{W}}_N^{\Gamma_1}+
\delta_T^{\Gamma_2}\un{\hat{W}}_N^{\Gamma_2}=\begin{pmatrix}
\overset{\circ}{R}_{T}(\phi_H^N \un{e}_x)\\
\overset{\circ}{R}_{T}(\phi_H^N \un{e}_y) \end{pmatrix}
\end{aligned}\right.
\end{equation}
{The linear equations above only involves the $(\un{\hat{W}}_N^{\Gamma_i})_i$ unknowns associated to vertex $N$ (and it is the unique system in which they appear).} 
The nature of the linear system depends on the position of the node {(see Figure~\ref{fig:patchesEET} for a simple illustration of the various cases)}. For internal nodes and nodes on the Dirichlet boundary, the resulting system of equations is under-determined, {an extra constraint is then introduced in order to determine} the components of the edge tractions radiating from the node (the constraint is detailed in Section~\ref{sec:StarfleetEET}). 

\begin{figure}
\centering
\includegraphics[width=.6\textwidth]{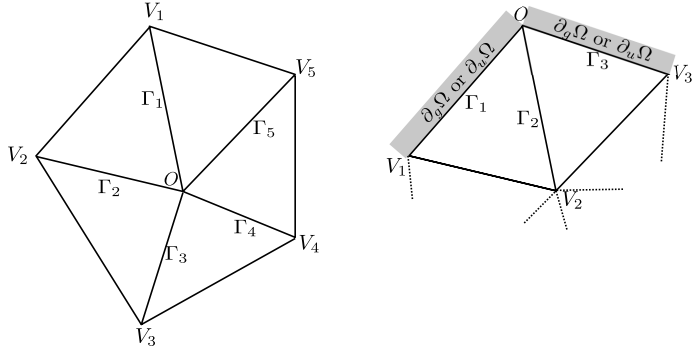}\caption{Different star-patch configurations to determine $\un{\hat{W}}_O^{\Gamma_i}$}\label{fig:patchesEET}
\end{figure}

Once those local systems have been solved, one has to compute the balanced traction forces from the fluxes  ${\hat{W}}_N^\Gamma$. For an edge $\Gamma$ of vertexes $l$ and $r$, the balanced traction forces are decomposed over the shape functions basis and determined by the inversion of the mass matrix of the edge:
\begin{equation}
 \tF =  \tF_l \phi_H^l + \tF_r \phi_H^r \text{ and } \begin{pmatrix}
\int_\Gamma \phi_H^l  \cdot \phi_H^l \,dS & \int_\Gamma \phi_H^r  \cdot \phi_H^l\, dS\\ 
\int_\Gamma \phi_H^r  \cdot \phi_H^l\, dS & \int_\Gamma \phi_H^r  \cdot \phi_H^r \,dS
 \end{pmatrix}\begin{pmatrix}\tF_l\\\tF_r\end{pmatrix}=\begin{pmatrix} \un{\hat{W}}_l^\Gamma\\\un{\hat{W}}_r^\Gamma\end{pmatrix}
\end{equation}

\subsection{Element estimation of the complementary energy}\label{ssec:classicEc}

Once $(\tF)$ have been determined, they can be used as Neumann loading on each element {to determine the statically admissible stress fields $\ts$ independently on each element \eqref{eq:introF}}. For simple cases of volume loading, the exact solution for the stress field can be obtained using a piecewise polynomial search space. {For the general cases, a numerical approximation of high precision must be obtained. A dual approach would directly give the desired stress field but classically, in order to keep using widespread tools, a refined finite element approximation in displacement is sought, characterized by Finite Element space $\KAp^T$.}
The precise FE space can be obtained by higher order approximation (p version), $3$ degrees higher than the original computation is known to be sufficient \citep{Bab94}, or by remeshing the elements (h version). 

{Then in order to determine $\sigp\simeq\ts$ on Element $T$, we seek the displacement  $\depp$ solution of the following system:}
\begin{equation}\label{eq:Eeq}\left\{\begin{aligned}
&\int_T \sigp:\strain{\un{\Phi}} d\domain -\left( \int_T \un{f}\cdot\un{\Phi} d\domain  + \sum_{\Gamma \in \bound T} \int_\Gamma \delta_T^\Gamma \tF \cdot \un{\Phi}\, dS \right) = 0,\ \forall \un{\Phi}\in\KAp^T\\
&\sigp=\hooke:\strain{\depp},\quad \depp \in \KAp^T
\end{aligned}\right.
\end{equation}

Note that if the space $\KAp^T$ was not sufficiently large then the energy would be
underestimated, leading to a erroneously low error estimation (the bound would not be strict). 

%% file: newsa.tex
The driving ideas for our new SA-field reconstruction technique are:
\begin{itemize}
\item To simplify the implementation (and somehow accelerate the execution) by avoiding loops and tests, by vectorizing the code, by working on natural finite element containers (elements instead of star-patches) prone to small grain parallelism.
\item To separate the topological properties and the geometrical properties of the mesh to make best use of both.
\item To study all the consequences of the strong prolongation equation.
\item To derive the classical EET as a special case, propose new variants including one which makes minimal the error estimator $\ecr{\depH,\ts}{\domainH}$.
\end{itemize}\medskip

To do so,  we add the following ingredients to the two-step procedure:
\begin{itemize}
\item In order to define the Neumann conditions on the borders of elements:
\begin{itemize}
\item We assume no a priori shape for the edge tractions, and we only search for their work in a well chosen basis $\basis$ of $\KAH$ (Section~\ref{ssec:worknotforce}). 
\item We use all known data on the boundary (Neumann's given forces and Dirichlet's computed nodal reactions), so that the only unknowns are the works of tractions on the internal edges (Section~\ref{ssec:elimextern}).
\item We gather unknown works in multivectors\footnote{A multivector is a collection of vectors, or equivalently a thin matrix, prone to block treatment.} and we write the linear system they must be solution to:
\begin{equation}
\matG \,\vecW\!(\basis) = \tvecR\!(\basis)
\end{equation}
$\matG$ contains topological data related to the strong prolongation equation and geometrical information for the unknown works to respect linearity principle. Multivector~$\tvecR(\basis)$ stands for balance residuals  computed on the elements, Multivector~$\vecW(\basis)$ stands for the works computed on the edges (Section~\ref{sec:proloG}).
\item We find one solution $\vecW_0(\basis)$ to previous system (Section~\ref{sec:solG}) and a basis $\kerG$ of $\operatorname{ker}(\matG)$ (Section~\ref{sec:kerG}), so that for any $\gamma$,  $\vecW\!(\basis)= \vecW_0(\basis) + \kerG \gamma$ satisfies the strong prolongation. 
\item We choose one criterion and search for optimal $\gamma$ (Section~\ref{sec:optimW}).
\end{itemize}
\item For the element estimation of the complementary energy:
\begin{itemize}
\item We use a higher degree basis $\basis^2$ and we compute the associated works of Neumann boundary conditions (Section~\ref{sec:highprec}).
\end{itemize}
\end{itemize}

\subsection{Choice of the unknowns and associated constraints}
\label{ssec:worknotforce}
The first point of the method is that no a priori shape is assumed for the edge tractions (whereas they are searched for as affine functions in the classical EET). Indeed the strong prolongation equation \eqref{eq:prolo_resid} only imposes linear constraints on their work in kinematically admissible fields. We then search for works $(\hat{W}^\Gamma(\un{v}))$ for all edges $\Gamma$ and all $\un{v}\in\KAH$
\begin{equation}
\hat{W}^\Gamma(\un{v}) =\int_\Gamma \tF\cdot\un{v} \,dS
\end{equation}
For practical computation, a basis of $\KAH$ must be chosen. In the classical EET, shape (hat) functions are chosen. Here, in order to avoid the tedious forming of star-patches and enable systematic treatment, we use the canonical basis $\basis=(1\un{e}_x,x\un{e}_x,y\un{e}_x,1\un{e}_y,x\un{e}_y,y\un{e}_y)$ on each element which is indeed independent of the geometry of the mesh. 

Note that the unknown works are not independent. The works being linear forms on $\KAH$, the variables $x$ and $y$ being affinely linked on an edge, we must have for an arbitrary direction $\un{e}$:
\begin{equation}
\begin{aligned}
a_\Gamma \hat{W}^\Gamma(x\un{e})+b_\Gamma \hat{W}^\Gamma(y\un{e})+c_\Gamma\hat{W}^\Gamma(\un{e}) =\hat{W}^\Gamma \left((a_\Gamma   x+ b_\Gamma y + c_\Gamma)\un{e} \right)=0 
 \end{aligned}
\end{equation}
where $(a_\Gamma,b_\Gamma,c_\Gamma)$ are the  coefficients of the equation of Edge~$\Gamma$:
\begin{equation*}
 (x,y) \in\Gamma \Longleftrightarrow\ a_\Gamma x + b_\Gamma y + c_\Gamma = 0   
\end{equation*}
In matrix form, this constraint writes:
\begin{equation}\label{eq:geomconsist}
\begin{pmatrix}
c_\Gamma & a_\Gamma & b_\Gamma 
\end{pmatrix} 
\begin{pmatrix}
 \hat{W}^\Gamma(\un{e}_x) &  \hat{W}^\Gamma(\un{e}_y)\\
 \hat{W}^\Gamma(x\un{e}_x) &  \hat{W}^\Gamma(x\un{e}_y)\\
\hat{W}^\Gamma(y\un{e}_x) &  \hat{W}^\Gamma(y\un{e}_y)
\end{pmatrix} = \begin{pmatrix}
0 & 0 
\end{pmatrix} 
\end{equation}
where we see that the uncoupling of the directions $\un{e}_x$ and $\un{e}_y$ leads to working with a multivector.

\begin{prop}\label{prop:lineartraction}
The necessary condition~\eqref{eq:geomconsist} is sufficient to build a traction force $\tF\cdot\un{e}$ associated with the works $(\hat{W}^\Gamma(\un{e}),\hat{W}^\Gamma(x\un{e}),\hat{W}^\Gamma(y\un{e}))$. 
\end{prop}
\begin{proof} {An infinity of traction forces can be associated with the given works satisfying~\eqref{eq:geomconsist}, among them we prove that there exists one linear traction force ; other traction fields could be obtained by adding a term  orthogonal to the subspace of first degree polynomials in the $L^2(\Gamma)$ sense.}

Let us consider the edge $\Gamma$ ended by Vertexes $A$ and $B$ of coordinates $(x_A,y_A)$ and $(x_B,y_B)$. We compute the coefficient of the line and its middle~$O$:
\begin{equation}
\begin{aligned}
&a^\Gamma = y_B - y_A \\
&b^\Gamma = x_A - x_B \\
&c^\Gamma = x_By_A - y_Bx_A
\end{aligned}\qquad 
\begin{aligned}
&x_O=\frac{x_A + x_B}{2}\\
&y_O=\frac{y_A + y_B}{2}\\
\end{aligned}
\end{equation}
We introduce the curvilinear abscissa $s\in[-1/2,1/2]$, we then have:
\begin{equation}
x=x_O-s b^\Gamma,\quad y=y_O+s a^\Gamma\quad\text{and}\quad \,dS= \meas(\Gamma)ds
\end{equation}
If we search for a linear traction force in direction $\un{e}$, parametrized by coefficients $F^\Gamma_0$ and $F^\Gamma_1$ such that $\tF\cdot\un{e}=F^\Gamma_0+s F^\Gamma_1$. A simple integration over the edge shows that:
\begin{equation}
\begin{aligned}
\meas(\Gamma) F^\Gamma_0 &= \hat{W}^\Gamma(\un{e})\\
\frac{\meas(\Gamma)^3}{12} F^\Gamma_1&= a^\Gamma \hat{W}^\Gamma(y\un{e})-b^\Gamma \hat{W}^\Gamma(x\un{e})-(a^\Gamma y_O - b^\Gamma x_O)\hat{W}^\Gamma(\un{e})
\end{aligned}
\end{equation}
\end{proof}

\subsection{Elimination of external edges, definition of the right-hand-side}
\label{ssec:elimextern}

A great interest of using works instead of tractions is that they are fully known on the border edges: even on Dirichlet edges they can be deduced from the nodal reactions on Dirichlet vertexes. 
These data can then be incorporated in the residual equation \eqref{eq:prolo_resid}, so that only works on internal edges are unknown. 

For a test field $\un{v}_H\in \KA_H$, we define the complete residual $R_T(\un{v}_H)$:
\begin{itemize}
\item If Element $T$ is an internal element then $R_T(\un{v}_H)=\overset{\circ}{R}_T(\un{v}_H)$
\item If Element $T$ is connected to one\footnote{\label{note:oneedge}Elements shall have at most one edge on the border of the mesh, see remark~\ref{rem:meshborder}.} Neumann edge $\Gamma_g$, then we set:
\begin{equation}
R_T(\un{v}_H) = \overset{\circ}{R}_T(\un{v}_H) - \int_{\Gamma_g} \un{g}\cdot \un{v}_H \, dS
\end{equation}
\item If Element $T$ is connected to one\textsuperscript{\ref{note:oneedge}} Dirichlet edge $\Gamma_d$, we use the computed nodal reactions:
\begin{equation}
R_T(\un{v}_H) = \overset{\circ}{R}_T(\un{v}_H) -  \sum_{V\in\bound \Gamma_d} \alpha_V^{\Gamma_d} \un{\lamN}_d(V) \cdot \un{v}_H(V)
\end{equation}
$V$ represents the vertexes of $\Gamma_d$, $\un{\lamN}_d(V)$ is the associated nodal value of the reaction, which is known from the FE resolution, and $\un{v}_H(V)$ is the value of the test field at the vertexes. $\alpha_V^\Gamma$ is a partition of the reaction force between the Dirichlet edges the vertex $V$ belongs to (see Figure~\ref{fig:dirichlet} for an illustration):
\begin{equation}
\forall V\in \partial_u\domainH,\ \sum_{\Gamma\in\boundi V\cap \partial_u\Omega} \alpha_V^\Gamma = 1
\end{equation}
Typically $\alpha_V^\Gamma$ can be set to $1/2$ for nodes inside Dirichlet boundary and $1$ for nodes on the extremities of Dirichlet boundary. Another possibility is to set $\alpha_V^\Gamma$ according to the length of the edges connected to Vertex $V$ (of course for regular meshes these approaches are equivalent). A more sophisticated technique would be to choose $\alpha_V^\Gamma$ in order to minimize a distance between an equivalent traction and $\sig_H\cdot\nga$ post-processed from the finite element computation.
\end{itemize}
\begin{figure}[ht]\centering
\includegraphics[width=0.9\textwidth]{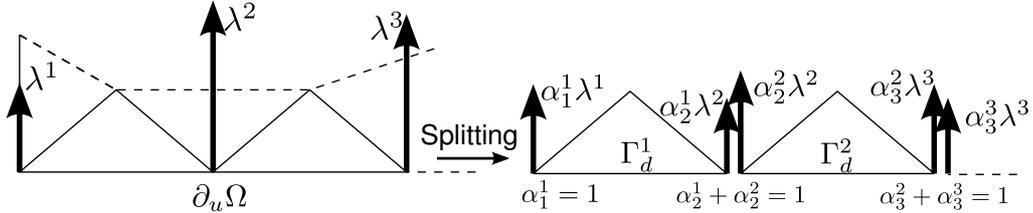}\caption{Splitting of reactions between adjacent elements}\label{fig:dirichlet}
\end{figure}

\subsection{Strong prolongation equation at the global scale}
\label{sec:proloG}

The strong prolongation equation~\eqref{eq:prolo_resid} can be rewritten using only internal edges works as unknowns:
\begin{equation}
\begin{aligned}\label{eq:prolo_resid_int}
\forall T\in\setelem,\ \forall \un{v}_H\in\KA_H,\ \sum_{\Gamma\in\bound T\cap\setedgei} \delta_T^\Gamma W^\Gamma(\un{v}_H) = R_T(\un{v}_H)
\end{aligned}
\end{equation}
where $\setedgei$ is the set of internal edges (edges separating two elements).

Equation~\eqref{eq:prolo_resid_int} writes in matrix form:
\begin{equation}\label{eq:strongprologlo}
\forall \un{v}_H\in\KA_H,\ \matD \vecW(\un{v}_H) = \vecR(\un{v}_H)
\end{equation}
where 
\begin{equation*}
\begin{aligned}
&\vecW(\un{v})= \begin{pmatrix} \vdots \\ W^\Gamma(\underline{v}) \\\vdots\end{pmatrix} \in\R^{|\setedgei|}
\quad,\quad
\vecR(\un{v})= \begin{pmatrix} \vdots \\ R_T(\underline{v}) \\\vdots\end{pmatrix} \in\R^{|\setelem|}\\
&  \text{and } \matD \text{ is the } |\setelem|\times|\setedgei| \text{ matrix of coefficients } (\delta_T^\Gamma)
\end{aligned}
\end{equation*}

Using the canonical basis\footnote{In fact, a translation and a scaling of the basis can be interesting to improve conditioning. Typically the test functions shall span $[-1,1]$ when describing the whole domain. This corresponds to well choosing the origin and the scale of the frame.} on each element in order to span $\KA_H$, the constraints~\eqref{eq:geomconsist} and the strong prolongation equation~\eqref{eq:strongprologlo} write:
\begin{equation}\label{eq:todo1}
\begin{pmatrix}
\matD & &   \\
& \matD &   \\
& & \matD   \\
\matc & \mata  & \matb 
\end{pmatrix}
\begin{pmatrix}
\vecW(\un{e}_x) &  \vecW(\un{e}_y)\\\vecW(x\un{e}_x) & \vecW(x\un{e}_y)\\ \vecW(y\un{e}_x) &   \vecW(y\un{e}_y)
\end{pmatrix}=
\begin{pmatrix}\vecR(\un{e}_x)&\vecR(\un{e}_y) \\ \vecR(x\un{e}_x)& \vecR(x\un{e}_y) \\ \vecR(y\un{e}_x)& \vecR(y\un{e}_y) \\ 0 & 0
\end{pmatrix}
\end{equation}
where $(\matc,\mata,\matb)$ are the diagonal matrices containing the geometrical coefficients $(c^\Gamma,a^\Gamma,b^\Gamma)_{\Gamma\in\setedgei}$ of the condition~\eqref{eq:geomconsist}.

This system is $(3\card{\setelem}+\card{\setedgei})\times 3\card{\setedgei}$ large but it is extremely sparse, in fact it is never assembled, and it possesses many useful properties for fast solving.

In order not to interrupt the presentation, the study and the resolution of the system are explained in Section~\ref{sec:solving}. The main result is that the system has many solutions, so that a criterion can be added to choose one solution; we propose various criteria, among others one which recovers the classical EET estimator and one which minimizes the error estimation~\eqref{eq:computeerror2}.

\subsection{Element estimation of the complementary energy}\label{sec:highprec}
We note $\left(\vecW(f\un{e})\right)$ ($f\in\{1,x,y\}$ and $\un{e}\in\{\un{e}_x,\un{e}_y\}$) the chosen solution to System~\eqref{eq:todo1}. In order to compute the strain energy inside elements, we need to solve with high precision the following system on each element~$T$:
\begin{equation}\label{eq:hdelem}
\begin{aligned}&\text{Find }\dep_p\in\KAp^T,\text{ such that }  \forall \depv_p\in\KAp^T,\\
& \int_T \strain{\dep_p}:\hooke:\strain{\depv_p} d\Omega= \int_T \un{f} \cdot \depv_p d\Omega+\sum_{\Gamma\in\bound T} \delta_T^\Gamma W^\Gamma(\depv_p) 
\end{aligned}
\end{equation}

We propose to use a p-refinement strategy to define $\KAp^T$, and in practice to use a higher degree canonical basis, for instance for order~2:
\begin{equation}
\mathcal{B}^2=(x\un{e}_x, y\un{e}_y, x\un{e}_y+y\un{e}_x, xy\un{e}_x, xy\un{e}_y, x^2\un{e}_x, x^2\un{e}_y, y^2\un{e}_x, y^2\un{e}_y)
\end{equation}
\begin{remark}
The infinitesimal rigid body motions $(1\un{e}_x,1\un{e}_y,x\un{e}_y-y\un{e}_x)$ have been omitted, so that the element problems are directly well posed.
\end{remark}

First degree edge works, associated with test fields $(x\un{e}_x, y\un{e}_y, x\un{e}_y+y\un{e}_x)$, are known from~\eqref{eq:todo1}.
 We need an extrapolation strategy for the determination of the higher degree works, associated with $(xy, x^2, y^2)$ in both directions. These higher degree works need to remain consistent with the geometry and the first order works, typically we must have, for any edge $\Gamma$ and direction $\un{e}$:
\begin{equation}
\begin{aligned}
&a_\Gamma W^\Gamma(x^2\un{e}) +b_\Gamma W^\Gamma(xy\un{e})+c_\Gamma W^\Gamma(x\un{e}) = 0  \\
&a_\Gamma W^\Gamma(xy\un{e}) +b_\Gamma W^\Gamma(y^2\un{e})+c_\Gamma W^\Gamma(y\un{e}) = 0
\end{aligned}
\end{equation}
 
The simplest solution to build these higher degree works is to suppose they are associated with a linear distribution of traction force per edge. As proved in Proposition~\ref{prop:lineartraction}, the linear traction is uniquely determined by the first degree works, and using the notations of that proposition, we have for degree 2:
\begin{equation}\label{eq:hdW}
\begin{aligned}
\vecW(x^2\un{e})&= \left(\frac{\matb^2}{12}-\matx_O^2\right)\vecW(1\un{e})+2\matx_O\vecW(x\un{e})\\
\vecW(xy\un{e})&=\left(-\frac{\mata\matb}{12}-\matx_O\maty_O\right)\vecW(1\un{e})+\matx_O\vecW(y\un{e})+\maty_O\vecW(x\un{e})\\
\vecW(y^2\un{e})&=\left(\frac{\mata^2}{12}-\maty_O^2\right)\vecW(1\un{e})+2\maty_O\vecW(y\un{e})
\end{aligned}
\end{equation}
where $(\matx_O,\maty_O)$ are the diagonal matrices containing the coordinates of the middle of internal edges. As the vector notation shows, the right hand side of all elements computations are obtained at the same time. After that, the high degree element problems~\eqref{eq:hdelem}, and the contribution to the error estimator~\eqref{eq:computeerror2} can be computed in parallel (typically on multiple cores).

%% file: solving.tex
In this section, we present how System~\eqref{eq:todo1} is solved, and how its solution can be optimized. We use the following notations:
\begin{equation}
\matG=\begin{pmatrix}
\matD & &   \\
& \matD &   \\
& & \matD   \\
\matc & \mata  & \matb 
\end{pmatrix} \quad\text{and}\quad \vecW(\basis)=\begin{pmatrix}
\vecW(\un{e}_x) &  \vecW(\un{e}_y)\\\vecW(x\un{e}_x) & \vecW(x\un{e}_y)\\ \vecW(y\un{e}_x) &   \vecW(y\un{e}_y)
\end{pmatrix}  
\end{equation}

We recall that $\setedgei$ is the set of internal edges (edges separating two elements) which we oppose to border edges $\setedgee$. $\setvertexi$ and $\setvertexe$ are respectfully the sets  of internal and boundary vertexes. Of course $\card{\setedge}=\card{\setedgei}+\card{\setedgee}$ and $\card{\setvertex}=\card{\setvertexi}+\card{\setvertexe}$.
\begin{remark}\label{rem:meshborder}
In order to be a part of a valid mesh for finite element approximation, an element can not have two edges on the skin of the domain \citep{george91}, which implies that  $\card{\setedgee}=\card{\setvertexe}$. 
\end{remark}

We also recall that elements, edges and vertexes satisfy the Euler identity:
\begin{equation}
\card{\setelem}-\card{\setedge}+\card{\setvertex}=c-h
\end{equation}
where $c$ is the number of connected components and $h$ is the number of holes within the structure. In the following, we will be working on one single connected component so that $c=1$.

\subsection{Properties of Matrix $\matD$}\label{sec:studyD}

{The matrix $\matD$ which gathers the coefficients $(\delta_T^\Gamma)$ is} rectangular with $\card{\setelem}$ rows and $\card{\setedgei}$ columns.  It is a very sparse signed boolean matrix, since a row contains at most 3 non-zero coefficients (2 for a row associated with an element on the boundary), and each column contains exactly 2 non-zero coefficients (with opposite sign if all elements have the same orientation). 

This matrix describes how elements are connected by their edges, in algebraic topology it is called an incidence matrix.

\begin{remark}The coefficients $(\delta_T^\Gamma)$ need to be chosen in agreement with the orientation of elements and edges. For 2D meshes, it simply consists in giving edges an initial and a final vertex, and giving triangles a rotation sense (clockwise or counter-clockwise); if the orientations match then coefficient $\delta_T^\Gamma$ is $+1$. For the future developments we will assume that all elements are given the same orientation, this is not a necessity, it will only simplify the writing of the vector $\veci$ in Equation~\eqref{eq:leftker}. Note that orientation is a topological property which can be determined by the sole analysis of the ordering of the list of vertexes given to describe the edges and the elements, but it can be helpful to use geometrical computations (like triple products) to determine it (in particular in 3D).
\end{remark}

\begin{prop}
Matrix $\matD$ has more columns than rows and it is not full row-ranked.
Anyhow, the problem \eqref{eq:strongprologlo} remains well posed.
\end{prop}
\begin{proof}
Assuming all elements have the same orientation, all internal edges are followed positively by one element and negatively by the other, the left kernel is then constituted by the vector $\veci=\begin{pmatrix} 1&\ldots&1\end{pmatrix}^T$.

The problem \eqref{eq:strongprologlo} remains well posed because the finite element equilibrium over the whole domain is equivalent to the orthogonality between the right hand side and the left kernel:
\begin{equation}\label{eq:leftker}
\begin{aligned}
\veci^T\matD&=0 \\
\veci^T\vecR(\un{v}_H) &= \sum_T R_T(\un{v}_H) = \int\limits_{\Omega} (\sigH:\strain{\un{v}_H}-\un{f}\cdot\un{v}_H ) d\Omega - \int\limits_{\partial_g\Omega}  \un{g}\cdot\un{v}_H dS = 0
\end{aligned}
\end{equation}
\end{proof}

\begin{prop} The dimension of the right kernel of $\matD$ is  $\card{\setvertexi}+h$ and a basis $\matN$ can be computed from the analysis of the mesh; moreover we have $\matN=[\matN_V,\matN_h]$, where $\matN_V$ is associated with internal vertexes and $\matN_h$ is associated with holes.
\end{prop}
\begin{proof}
The rank of $\matD$ is $(\card{\setelem}-1)$ and then the dimension of its right kernel is $(\card{\setedgei}-\card{\setelem}+1)$.  Using Remark~\ref{rem:meshborder}, the dimension of the right kernel of matrix $\matD$ is then:
\begin{equation*}
\card{\setedgei}-\card{\setelem}+1 =\card{\setedge}-\card{\setedgee}-\card{\setelem}+1= \card{\setvertex}- \card{\setedgee} = \card{\setvertexi}+h
\end{equation*}

It is easy to build a basis of the right kernel by realizing it corresponds to closed paths within the mesh (see Figure~\ref{fig:delta_ker}): 
\begin{itemize}
\item Each internal node defines a star-patch with internal edges as branches. When combining the edges, all elements of the star-patch are counted twice with opposite signs.
\item The internal edges radiating from nodes on the boundary of one hole define a (closed) loop that surrounds the hole.
\end{itemize}
\end{proof}
\begin{figure}[ht]\centering
\includegraphics[width=.35\textwidth]{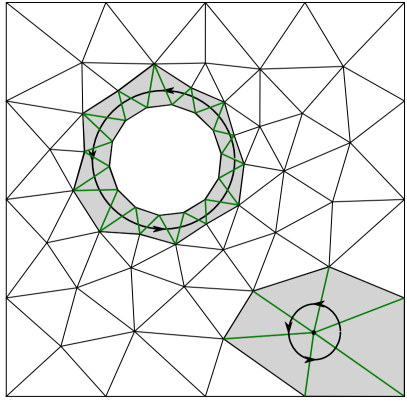}\caption{Matrix~$\matD$ kernel: internal node star patch and hole turns}\label{fig:delta_ker}
\end{figure}
Note that this basis can be efficiently obtained using tools from graph analysis or, with higher complexity using algebraic topology \citep{rapetti03}. 

\subsection{Study of $\matG$}\label{sec:kerG}
\begin{prop} Let $\matx_V$ and $\maty_V$ be the diagonal matrices of the coordinates of the internal vertexes, then the right kernel~$\kerG$ of~$\matG$ writes:
\begin{equation}
\kerG= \begin{pmatrix} \matN_V \\  \matN_V \matx_V  \\ \matN_V\maty_V  \end{pmatrix}
\end{equation}
\end{prop}
\begin{proof}
The components of $\operatorname{ker}(\matG)$  need to be sought in subspaces of $\operatorname{ker}(\matD)$.
Consider one star-patch centered on vertex $V$ of coordinates $(x_V,y_V)$, let $\Gamma$ be a branch of extremity $A^\Gamma=(x_A,y_A)$. 
Of course, the center $V$ being on the edge, we have $a_\Gamma x_V + b_\Gamma y_V + c_\Gamma=0$ independently on $\Gamma\in\boundi V$. Let $N_V$ be the null mode of $\matD$ associated with~$V$. Then $\begin{pmatrix}
 N_V  ^T & 
 x_V N_V ^T & 
 y_V N_V ^T
\end{pmatrix}^T$
belongs to the kernel of $\matG$.

All internal nodes are thus associated with one kernel vector. Holes can not be associated with null modes of the matrix $\matG$ because not all rays radiating from one hole can be concurring to the same point.
\end{proof}

\begin{prop} The dimension of the left kernel of $\matG$ is equal to the number of vertexes of the mesh:
\begin{equation}
\operatorname{dim}(\ker(\matG^T))=\card{\setvertexi} +\card{\setvertexe} =\card{\setvertex}
\end{equation}
It is indeed possible to create a basis vector of $\ker(\matG^T)$ associated to each vertex of the mesh.
\end{prop}
\begin{proof}
The result is a combination of the rank theorem applied to matrix $\matG^T$, and of the relation $3\card{\setelem}-(2\card{\setedgei}+\card{\setedgee})=0$ which corresponds to the fact that each element is made out of 3 edges whereas each internal edge belongs to 2 elements and each external edge belongs to one element.

For the construction of a basis vector, we can restrict our analysis to elements and edges involved in a star-patch. 
Let $O$ be a vertex (internal or external), the elements in the star-patch are the $\setelem_{SP}=\boundii O$. {The edges involved in the star-patch are the rays $\setedge_{SP,r} = \boundi O$ and the borders $\setedge_{SP,b} = \bound (\boundii O )$. The set of internal edges involved in the star-patch is then  $\setedgei_{SP}=\left(\setedge_{SP,r}\bigcup\setedge_{SP,b}\right) \setminus \setedgee$.}

We can isolate the part $\tilde{\matG}$ of $\matG^T$ concerned by these components when searching its kernel:
\begin{equation}
\begin{pmatrix}
\tilde{\matD}^T& & & \tilde{\matc} \\
&\tilde{\matD}^T & & \tilde{\mata} \\
&&\tilde{\matD}^T &  \tilde{\matb}
\end{pmatrix}
\begin{pmatrix}
\gamma \\
\alpha \\
\beta \\
\mu
\end{pmatrix} = 0
\end{equation}
where $\tilde{\matD}^T$ (respectively $\tilde{\matc},\ \tilde{\mata},\ \tilde{\matb}$) is the submatrix of $\matD^T$ (respectively $\matc,\ \mata,\ \matb$) related to the star-patch of size $\card{\setedgei_{SP}}\times\card{\setelem_{SP}}$ (respectively diagonal matrices of size $\card{\setedgei_{SP}}$).
$(\alpha,\beta,\gamma)$ are vectors associated with elements, $\mu$ is associated with edges. We note $\tilde{\matl}$ the diagonal matrix of the length of edges in  $\setedgei_{SP}$, so that $0<\tilde{\matl}^2=\tilde{\mata}^2+\tilde{\matb}^2$.

{The border edges $\setedge_{SP,b}$ (including the ones on the border of $\domain$) are always associated to one (and only one) element of the star-patch}, so that we number border edges and elements accordingly, see Figure~\ref{fig:patchesEET2}.

The reader can verify that the following vector belongs to the kernel of~$\tilde{\matG}^T$ (next proposition gives an important interpretation):
\begin{equation}
\left\{\begin{aligned}
\alpha&=\mate_b^{-1}\mata_b\\
\beta&=\mate_b^{-1}\matb_b\\
\gamma&=\mate_b^{-1}\matc_b\\
\mu&=-\tilde{\matl}^{-2} \left(\tilde{\mata}\tilde{\matD}^T\alpha +\tilde{\matb}\tilde{\matD}^T\beta\right)
\end{aligned}\right. \quad \text{ where } \mate_b=x_O\mata_b+y_O\matb_b+\matc_b
\end{equation}
$\mate_b$ is a diagonal matrix of non-null coefficients since the center of the star-patch can not be aligned with one of the border edge.
\end{proof}

\begin{figure}
\centering
\includegraphics[width=.6\textwidth]{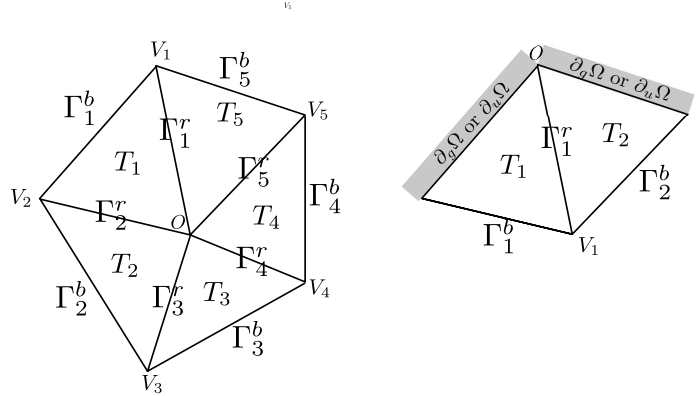}\caption{Star-patches, numbering of elements, border edges (b) and radial edges (r)}\label{fig:patchesEET2}
\end{figure}

\begin{prop}
The strong prolongation equation~\eqref{eq:todo1} has solutions.
\end{prop}
\begin{proof}
We just need to verify that the right-hand-side is orthogonal to the kernel of $\tilde{\matG}^T$. If we use one vector of the basis built in previous expression and apply it to the right hand side, we in fact compute the work of the residual in one FE shape (hat) function, which is zero by the definition of the finite element approximation (Galerkin orthogonality). Indeed the characteristic equation of a border edge, restricted to the adjacent element, is exactly the slope of the hat function; the normalizing term $\mate_b^{-1}$ ensures that all the slopes are worth 1 at the vertex $O$.
\end{proof}


\subsection{Efficient solving}\label{sec:solG}
Previous analysis proves that the rectangular system \eqref{eq:todo1}, though it is rank deficient in rows and columns, remains well posed. Because solutions are defined up to a member of $\operatorname{ker}(\matG)$ of which we know one basis ($\kerG$), the resolution is separated in two steps: this subsection addresses the question of finding one solution, the next subsection addresses the question of choosing an optimal contribution in the kernel of $\matG$ (for a chosen criterion).

Matrix $\matD$ is the algebraic description of the topological properties of circuits of edges seen as borders of elements: summation corresponds to union, multiplication by $(-1)$ corresponds to the reversal of the orientation. This structure implies the existence of the following Smith normal form for matrix $\matD$:
\begin{equation}
\begin{aligned}
\exists (\matU,\matV) & \text{ invertible signed boolean matrices with signed boolean inverses}\\
\text{such that}
\\
\matU\matD\matV&= \begin{pmatrix} 
1 & 	  & 	& 0&\hdots&0\\
  &\ddots & 	& \vdots&0&\vdots\\
  &       & 1 	& 0&&0\\
0 & \hdots & 0 	& 0 & \hdots & 0  
\end{pmatrix} \qquad \text{ with } \begin{aligned}\matU&=[\tilde{\matU}^T, \veci]^T \\\matV&=[\tilde{\matV},\matN]\end{aligned}
\end{aligned}
\end{equation}
From a practical point of view, this factorization can be easily computed by a full-pivoting Gauss procedure. Note that because $\matD$ is very sparse, the complexity of the procedure is linear (the position of pivots is always known a priori). 
 Note also that the left $(\veci)$ and right $(\matN)$ kernels of $\matD$ appear in the factorization.

In order to simplify notations, the columns of each multivector are gathered with the notation $\un{e}$, for instance $\vecR(1\un{e}):{=} \begin{pmatrix}\vecR(1\un{e}_x) &\vecR(1\un{e}_y) \end{pmatrix}$. We apply the Smith transform to system \eqref{eq:todo1}, and we split the works in two parts: $\vecW(f\un{e})=\begin{pmatrix}\tilde{\matV}&\matN\end{pmatrix} \begin{pmatrix}\alpha_{\un{e}}^f&\beta_{\un{e}}^f\end{pmatrix}^T$ (for $f=1,x,y$).
\begin{equation}
\begin{pmatrix}
\matI & \matO & & & & & \\
0& 0& & & & &  \\
& & \matI & \matO& &  \\
& & 0 & 0 & &  \\
& & & & \matI & \matO \\
& & & & 0 & 0\\
\matc\tilde{\matV} & \matc\matN &\mata\tilde{\matV}& \mata\matN  &\matb \tilde{\matV}& \matb \matN
\end{pmatrix}
\begin{pmatrix}
\alpha^1_{\un{e}} 
\\\beta^1_{\un{e}}
\\\alpha^x_{\un{e}}
\\\beta^x_{\un{e}} 
\\\alpha^y_{\un{e}}
\\\beta^y_{\un{e}}
\end{pmatrix}=
\begin{pmatrix}
\tilde{\matU}\vecR(1\un{e}) \\ 0 \\
\tilde{\matU}\vecR(x\un{e}) \\ 0 \\
\tilde{\matU}\vecR(y\un{e}) \\ 0 \\ \matO
\end{pmatrix}
\end{equation}
Previous equation determines the $(\alpha^f_{\un{e}})$ and all that remain to solve is: 
\begin{equation}
\begin{pmatrix}
 \matc\matN & \mata\matN  & \matb \matN
\end{pmatrix} \begin{pmatrix}
\beta^1_{\un{e}}\\\beta^x_{\un{e}} \\\beta^y_{\un{e}}
\end{pmatrix}=-\begin{pmatrix} \matc \tilde{\matV}\tilde{\matU}\vecR(1\un{e})+
\mata\tilde{\matV}\tilde{\matU}\vecR(x\un{e})+\matb \tilde{\matV}\tilde{\matU}\vecR(y\un{e})
\end{pmatrix}
\end{equation}
The matrix on the left hand side is $\card{\setedgei}\times3(\card{\setvertexi}+h)$. From previous analysis, we know the right kernel is $\card{\setvertexi}$ large. If we use the adapted internal vertex/hole basis for the left kernel, $\matN=[\matN_V,\matN_h]$, we even can tell that $\begin{pmatrix}\matc\matN_h&
  \mata\matN  & \matb \matN  \end{pmatrix}$ is full column-ranked. Another interest of the adapted basis for the kernel is that it maximizes the sparsity of the matrices (two star-patches have at most one edge in common, and one star-patch is in general made out of no more than 6 edges).

We then have to seek to the solution of the following full-column-ranked rectangular ($\card{\setedgei}\times (2\card{\setvertexi}+3h)$) though well-posed problem:
\begin{equation}
\begin{pmatrix}
\matc \matN_h& \mata\matN  & \matb \matN 
\end{pmatrix} \begin{pmatrix}
 \beta^h_{\un{e}} \\\beta^x_{\un{e}} \\\beta^y_{\un{e}}
\end{pmatrix}=-\begin{pmatrix}  \matc \tilde{\matV}\tilde{\matU}\vecR(1{\un{e}})+\mata\tilde{\matV}
\tilde{\matU}\vecR(x{\un{e}})+\matb \tilde{\matV}\tilde{\matU}\vecR(y{\un{e}})
\end{pmatrix}
\end{equation}
One possibility to solve that system is to use a LU factorization with pivoting in order to preserve sparsity: not only very few coefficients are non-zero but also the columns of $\mata\matN_V$ and $\matb \matN_V$ have the same fill-in (same remark holds for  $\mata\matN_h$, $\matb \matN_h$ and $\matc \matN_h$). 

As a result, we have determined at a low cost one solution matrix ($3\card{\setedgei}$ rows and 2 columns) to system \eqref{eq:todo1}:
\begin{equation}\label{eq:w0} \vecW_0(\mathcal{B})=\begin{pmatrix}
\tilde{\matV}\tilde{\matU}& & \\
& \tilde{\matV}\tilde{\matU}& \\
& & \tilde{\matV}\tilde{\matU}
\end{pmatrix}
\begin{pmatrix}\vecR(\un{e})\\ \vecR(x\un{e})\\ \vecR(y\un{e}) \end{pmatrix}
+
\begin{pmatrix}
\matN_h&&\\
&\matN&\\&&\matN\end{pmatrix}
\begin{pmatrix}
\beta^h_{\un{e}} \\
\beta^x_{\un{e}} \\
\beta^y_{\un{e}} \end{pmatrix}
\end{equation}

\subsection{Criteria to choose the element of the kernel}\label{sec:optimW}
Since we know the basis $\kerG$ of the right kernel of $\matG$, and one solution $\vecW_0(\mathcal{B})$ to  equation~\eqref{eq:todo1}, all works associated with stress fields satisfying the strong prolongation equation write  $\vecW(\mathcal{B})=\vecW_0(\mathcal{B})+\kerG\begin{pmatrix}\gamma_{\un{e}_x} & \gamma_{\un{e}_y}\end{pmatrix}$ where $(\gamma_{\un{e}_x} , \gamma_{\un{e}_y})$ are  vectors of size $\card{\setvertexi}$ which can be used in order to optimize the works with respect to a well chosen criterion.


Note that optimizing $(\gamma_{\un{e}_x} , \gamma_{\un{e}_y})$ is a mandatory step because the particular solution  $\vecW_0(\mathcal{B})$ was obtained with few references to the original mechanical problem, so that choosing $\gamma_x=\gamma_y=0$ leads in general to very large error estimation. 

In order to reconnect to the mechanical problem, the best option is to try to minimize the error estimate $\ecr{\depH,\ts}{\domainH}$, which in our case corresponds to minimizing the strain energy of the reconstructed SA-field. A less meaningful but much cheaper reference is the original finite element stress field $\sigH$.

In simplest cases, the criterion preserves the separation between the components, so that block resolution can be carried on. This is typically the case in the classical EET technique. In more general cases (like when we minimize the estimator), the components are coupled and must be sought together in a unique vector $\begin{pmatrix}\gamma_{\un{e}_x}^T & \gamma_{\un{e}_y}^T\end{pmatrix}^T$.

\subsubsection{Optimization with respect to the FE solution}

From the FE stress field $\sigH$, we can deduce an average traction $\FH$ and its work:
\begin{equation}
\begin{aligned}
\FH& = \sum_{T\in\boundi \Gamma} {\theta_T^\Gamma}  ({\sigH}_T\cdot\un{n}_T^\Gamma) \\
W_H^\Gamma(\un{\phi}_H) &= \sum_{T\in\boundi \Gamma} \int_\Gamma {\theta_T^\Gamma} ({\sigH}_T\cdot\un{n}_T^\Gamma)\cdot\un{\phi}_H \,dS
\end{aligned}
\end{equation}
where $ \theta_T^\Gamma$ is the ratio of the triangle's area over the area of the 2 triangles so that $\sum_{T\in\boundi \Gamma}  \theta_T^\Gamma=1$ (in classical EET $ \theta_T^\Gamma=\frac{1}{2}$).

Because stresses are constant inside elements, this average traction field is constant along the edge.  We note $\vecWnh_H(\mathcal{B})$ the vector of edge works associated with our test basis.

It is then possible to seek the work satisfying the strong prolongation which is the nearest (in the sense of a chosen norm) to the FE edge work:
\begin{equation}
\begin{aligned}
\begin{pmatrix}\gamma_{\un{e}_x} & \gamma_{\un{e}_y}\end{pmatrix} &= \arg\min_\mu \|\vecW_0(\mathcal{B})+\kerG\begin{pmatrix}\mu_{\un{e}_x} & \mu_{\un{e}_y}\end{pmatrix} - \vecWnh_H(\mathcal{B})\|_\mathcal{M}
\end{aligned}
\end{equation}
If we assume the matrix norm $\mathcal{M}$ treats the columns independently (like the Frobenius norm), and is represented by symmetric positive definite Matrix $\matM$, we have:
\begin{equation}\label{eq:optimnorm}
\begin{pmatrix}\gamma_{\un{e}_x} & \gamma_{\un{e}_y}\end{pmatrix}
\begin{aligned}
&= - \left(\kerG^T \matM \kerG \right)^{-1} \kerG^T \matM\left(\vecW_0(\mathcal{B})- \vecWnh_H(\mathcal{B})\right) 
\end{aligned}
\end{equation}
$\matM$ shall be chosen so that sparsity is preserved in order to make Cholesky factorization of the sparse $\card{\setvertexi}\times\card{\setvertexi}$ matrix  $\left(\kerG^T \matM \kerG \right)$  very fast. 

In the next section, assessments are given for the euclidean norm $\matM=\matI$ and the estimator based on this choice of minimization is referred to as \textsc{ starfleet}$_{\lVert\cdot\rVert_2}$. This estimator is given only because it is the fastest to apply, its performance are not expected to be good.

\subsubsection{Connection to classical EET}\label{sec:StarfleetEET}

In this subsection, we show that a well chosen norm corresponds exactly to the classical minimization within star-patches of the classical EET.

The EET criterion decouples the components $(\un{e}_x,\un{e}_y)$, so that the equation~\eqref{eq:optimnorm} applies and only one simple matrix $\matM$ needs to be defined. In order to simplify the writing of equations, we only work on one component which is not mentioned.

The EET seeks linear distributions of traction forces per edge which satisfy the strong prolongation condition. On one star-patch of center the internal node $N$, this problem is under-constrained, it is closed by the minimization of the distance $D_N$ between the work of the FE solution and of the SA-field in the shape function $\phi_N$:
\begin{equation}
D_N = \sum_{\Gamma \in\boundi N} \left(\int_\Gamma (\stF- \sFH)\phi_N  dS\right)^2
\end{equation}
When summing over all internal nodes, we obtain:
\begin{equation}
\sum_{N\in\setvertexi} D_N = \sum_{N\in\setvertexi}\sum_{\Gamma \in\boundi N} \left(\int_\Gamma  (\stF- \sFH)\phi_N  dS\right)^2 
\end{equation}
$\stF$ being searched for as a linear distribution, we can write
\begin{equation}
\stF = \stF_0 + \epsilon^\Gamma_N \stF_1(\phi_N - \frac{1}{2})
\end{equation}
where $\stF_0$ is the average value of $\stF$, and $\stF_1$ its variation along the edge (for a given orientation); $\epsilon^\Gamma_N$ is worth $+1$ on one vertex of $\Gamma$ (necessarily internal) and $-1$ on the other one (potentially on the boundary). Let $(x^\Gamma_0,y^\Gamma_0)$ be the middle of Edge~$\Gamma$, after integration on the edge, we get (the test field of the work is written as an index to shorten the expression, and $W^{\Gamma}_{H}$ stands for $W^{\Gamma}_{H}(1)$):
\begin{equation}
\begin{aligned}
&\sum_{N\in\setvertexi} D_N = \sum_{N\in\setvertexi}\sum_{\Gamma \in\boundi N} \left( \frac{1}{2}(\stF_0- \sFH)|\Gamma|+\frac{\epsilon^\Gamma_N}{12} \stF_1|\Gamma|\right)^2 \\ 
&= \sum_{N\in\setvertexi}\sum_{\Gamma \in\boundi N} \left(
\frac{\hat{W}^{\Gamma}_1 - W^{\Gamma}_{H}}{2} 
+ \frac{\epsilon^\Gamma_N}{2}\left( a_\Gamma (\hat{W}^\Gamma_y - y_0^\Gamma \hat{W}^\Gamma_1) - b_\Gamma (\hat{W}^\Gamma_x -  x_0^\Gamma \hat{W}^\Gamma_1) \right)\right)^2
\end{aligned}
\end{equation}
We can distinguish between edges whose vertexes are both internal ($\setedgeii$) and edges with one vertex on the boundary.
\begin{equation}
\begin{aligned}
&\sum_{N\in\setvertexi} D_N = \sum_{\Gamma \in \setedgeii} \frac{1}{2}\left(
\hat{W}^{\Gamma}_1 - W^{\Gamma}_H) \right)^2\!\! + \frac{1}{2}\left( a_\Gamma (\hat{W}^\Gamma_y - y_0^\Gamma \hat{W}^\Gamma_1) -b_\Gamma(\hat{W}^\Gamma_x -  x_0^\Gamma \hat{W}^\Gamma_1 )\right)^2 \\
 &+ \sum_{\Gamma \in \setedgei\setminus \setedgeii} \left(\frac{\hat{W}^{\Gamma}_1 - W^{\Gamma}_H}{2} 
+\frac{1}{2}\left( a_\Gamma (\hat{W}^\Gamma_y - y_0^\Gamma \hat{W}^\Gamma_1) - b_\Gamma (\hat{W}^\Gamma_x - x_0^\Gamma \hat{W}^\Gamma_1)\right)
\right)^2
\end{aligned}
\end{equation}
{
This criterion couples the works $(\hat{W}^{\Gamma}_1,\hat{W}^{\Gamma}_x,\hat{W}^{\Gamma}_y)$ on the same edge, this corresponds to matrix $\matM$ with only 3 non zero value per row.}

{
In the assessment, we propose a comparison, in term of CPU time, of the implementations of that criterion, using classical EET and using \textsc{Starfleet}.}

\subsubsection{Error estimator minimization}
This choice is of course the best possible but it is associated with much more computations and its interest is mostly theoretical. 

Because the two components $(\un{e}_x,\un{e}_y)$ are coupled by this criterion, we need to adopt a new notation, for a one-column vector of works:
\begin{equation*}
\begin{aligned}
\vvecW = \begin{pmatrix}
\vecW(\un{e}_x)^T & \vecW(x\un{e}_x)^T & \vecW(y\un{e}_x)^T &
\vecW(\un{e}_y)^T & \vecW(x\un{e}_y)^T & \vecW(y\un{e}_y)^T 
\end{pmatrix}^T \\ 
\un{\gamma} = \begin{pmatrix}\gamma_{\un{e}_x}\\\gamma_{\un{e}_x}\end{pmatrix} \qquad \un{\kerG} = \begin{pmatrix}\kerG&0\\0&\kerG\end{pmatrix}
\end{aligned}
\end{equation*}

As explained in Section~\ref{sec:highprec}, the works enable us to define high degree problems on elements, let $\stiff_{T}^p$ be the stiffness matrix associated with the precise approximation and $\matB_T$ be the matrix which defines the higher order boundary conditions on element $T$ from the edge work $\vvecW$. We need to minimize the strain energy:
\begin{equation}
\int_{\domainH} \ts:\hooke^{-1}:\ts d\Omega  = \frac{1}{2}\vvecW^T \left(\sum_{T\in\setelem} \matB_T^T {{\stiff^p}^{-1}_{T}} \matB_T\right)\vvecW
\end{equation}
which leads to the following results:
\begin{equation}
\un{\gamma} =  \left(\un{\kerG}^T\sum_{T\in\setelem} \matB_T^T {\stiff^p_{T}}^{-1} \matB_T\un{\kerG}\right)^{-1}\un{\kerG}^T\left(\sum_{T\in\setelem} \matB_T^T {\stiff^p_{T}}^{-1} \matB_T\right) \vvecW_0
\end{equation}
The computation of the element precise boundary conditions and stiffnesses being a necessity in any case, all costs are concentrated in the forming and factorization of the $2\setvertexi \times 2\setvertexi$ sparse symmetric matrix:
\begin{equation*} 
\left(\un{\kerG}^T\sum_{T\in\setelem} \matB_T^T {\stiff^p_{T}}^{-1} \matB_T\un{\kerG}\right)
\end{equation*}
 Despite the sparsity of this matrix (for well chosen kernel basis), the computational cost of this optimization is significant. It somehow corresponds to the solving of a dual problem of reduced dimension by primal ways. The estimator based on this choice of minimization is referred as \textsc{starfleet}$_{\lVert\cdot\rVert_{erdc}}$.



\bigskip
Algorithm~\ref{alg:EET/Starfleet}, summarizes the main steps of both \textsc{Starfleet} and  classical EET.

\begin{landscape}
\begin{algorithm2e}[ht]
\caption{EET (left) versus STARFLEET (right)}\label{alg:EET/Starfleet}
\begin{minipage}{0.46\linewidth}
$\bullet$ Construction of edge tractions $\tF$\;
\For{$V\in\setvertex$}{
Construction of the star-patch around $V$, detection of its type\;
\For{$T\in\boundii V$}{
Computation of the rhs \eqref{eq:prolo_resid}\;
}
Formation of the system corresponding to the strong prolongation equation \eqref{eq:sysEET_SP} if needed addition of extra constraint for its resolution (see Section~\ref{sec:StarfleetEET})\;
}
$\bullet$ Computation of $\ecr{\depH,\ts}{T}$ \;
\For{$T\in\setelem$}{
Construction and resolution of System \eqref{eq:Eeq}\;Computation of the element contribution to the error\;
}
\end{minipage}\hfill\vline\hfill
\begin{minipage}{0.46\linewidth }
$\bullet$ Construction of edge works $\hat{\mathbf{W}}$\;
Construction of matrices $\matD$, $\mata$, $\matb$, $\matc$\;
\For{$T\in\setelem$}{
Computation of the right-hand side $R_T(\basis)$ (Section~\ref{ssec:elimextern})\;
}
Computation of one solution $\hat{\mathbf{W}}_0(\basis)$,  Eq.~\eqref{eq:w0}\;
Optimization within $\operatorname{ker}(\matG)$ (Section~\ref{sec:optimW})\;
$\bullet$ Computation of $\ecr{\depH,\ts}{T}$\;
Definition of higher-degree works, Eq.~\ref{eq:hdW}\;
\For{$T\in\setelem$}{
Construction and resolution of System \eqref{eq:hdelem} in the canonical basis $\basis^2$\;
Computation of the element contribution to the error\;
}
\end{minipage}
\end{algorithm2e}
\end{landscape}

%% file: assessment.tex
\newcommand{\sfs}{SF}

In this section, the behavior of the estimate resulting from the construction of a SA-field presented above is analyzed for several mechanical problems. The FE computation and the construction of statically admissible stress fields are performed using an Octave code \citep{octave:2012}. For all the considered structures, the material is chosen to be isotropic, homogeneous, linear and elastic with Young's modulus $E=1$ and Poisson's ratio $\nu=0.3$.  We compare the new SA-field reconstruction (\textsc{starfleet} with $\|\|_2$ and $\|\|_{erdc}$ optimization) to the classical Element Equilibration Technique (which corresponds to a specific version of the \textsc{starfleet}) and Flux-free technique in term of quality of the estimators. In all cases, p-refinement is used for the local computations.
\textsc{starfleet} is shorten to \sfs\ in  figures and tables.

\subsection{Plane stress problem with analytical solution}
First, let us consider the rectangular domain $\Omega=]0, 8l[\times ]0,l[$, with Dirichlet homogeneous boundary conditions on $\partial \Omega$. The loading is a source term chosen such that the exact solution has the following analytical expression:
\begin{equation}
  \left\{
    \begin{aligned} 
\underline{u}\cdot \underline{e}_x&=x(x-8l)y(y-l)^3 \\
\underline{u}\cdot \underline{e}_y&=x(x-8l)y^2(y-l) 
\end{aligned}
  \right.
\end{equation}




Therefore, we can calculate the true error $e_{ex}$ through the exact energy norm:
\begin{equation}
e_{ex}=\sqrt{\lVert  \strain{\underline{u}_{ex}} \rVert_{ {\hooke}, \Omega}^2 -\lVert \strain{\depH} \rVert_{ {\hooke}, \Omega}^2}
\end{equation}
Figure~\ref{fig:poutre_p} and Table~\ref{tab:poutre_p} correspond to p+2 refinement applied to the EET, the flux-free, and to the \textsc{starfleet} (with various optimization). 
On that very regular problem, we observe that all estimators behave quite similarly.

\begin{figure}[ht]\centering
\begin{minipage}{.44\textwidth}
\begin{flushleft}
\includegraphics[width=.99\textwidth]{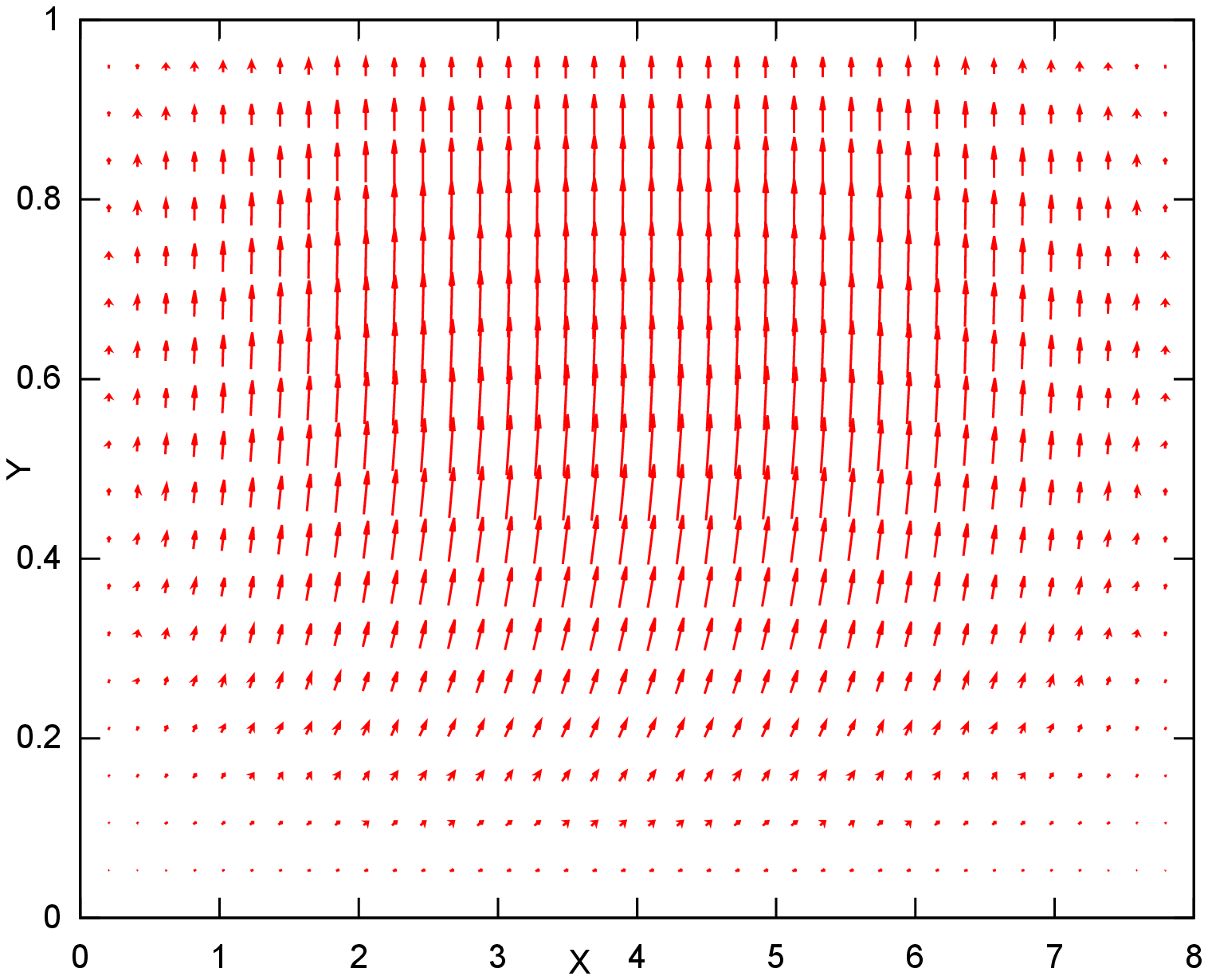}\caption{Displacement field of the problem with analytic solution}\label{fig:carrexact}
\end{flushleft}
\end{minipage}\hfill
\begin{minipage}{.55\textwidth}
\begin{flushright}
\begin{tikzpicture}
\begin{loglogaxis}[
scale=.94,
xlabel=$h$,
legend style={at={(.6,0.05)}, anchor=south west}] 

\addplot[color=blue,mark=diamond*] table[x=h,y=EET] {poutre_p.txt};
\addlegendentry{\begin{footnotesize}EET\end{footnotesize}}
\addplot table[x=h,y=STARFLEET] {poutre_p.txt};
\addlegendentry{\begin{footnotesize}{\sfs $_{\lVert\cdot\rVert_2}$}\end{footnotesize}}
\addplot table[x=h,y=STARFLEET2] {poutre_p.txt};
\addlegendentry{\begin{footnotesize}{ \sfs $_{\lVert\cdot\rVert_{erdc}}$}\end{footnotesize}}
\addplot table[x=h,y=SPET] {poutre_p.txt};
\addlegendentry{\begin{footnotesize}{ Flux-free}\end{footnotesize}}
\addplot[draw=violet, mark=triangle*] table[x=h,y=evraie] {poutre_p.txt};
\addlegendentry{\begin{footnotesize}$e_{ex}$\end{footnotesize}}
\addplot[color=black] coordinates {
(0.075, 0.5)
(0.075, 0.9375)
(0.04 , 0.5)
(0.075, 0.5)
};
\addlegendentry{\begin{footnotesize}h-slope\end{footnotesize}}
\end{loglogaxis}
\draw (-0.5,2.2) node[scale=1.,rotate=90]{Error};
\end{tikzpicture}\caption{Analytical problem: error estimators (p+2)}\label{fig:poutre_p}
\end{flushright}
\end{minipage}
\end{figure}

\begin{table}[ht]\centering
\begin{tabular}{|c|c|c|c|c|c|c|c|}
\hline 
ndof & $\frac{h}{l}$ &  EET & {\small \sfs$_{\lVert\cdot\rVert_2}$}  &\sfs$_{\lVert\cdot\rVert_{erdc}}$ & Flux-free 
&$e_{ex}$\\ 
\hline 
300 & 0.25  & 11.762& 11.498 & 11.129&  9.1987 
& 6.46715 \\ 
\hline 
1140& 0.125  & 5.0387 & 4.8102 & 4.584 & 4.3494 
& 2.9298\\ 
\hline 
4716 & 0.0625  & 2.3934 &2.2723 &2.2659  & 2.1957 
& 1.4436   \\ 
\hline 
10620 &0.04167  & 1.5274& 1.449 & 1.442 &1.4634  
& 0.9508  \\ 
\hline 
\end{tabular} 
\caption{Analytical problem, error estimators (p+2 version)}\label{tab:poutre_p}
\end{table}



\subsection{2D square in plane stress}
A square loaded with a horizontal unit shear force on the top boundary is considered. The displacement of the bottom is imposed to be zero and the two remaining sides are traction-free.


For this second case, we also performed a dual approach for the computation of the SA-field based on the construction of pure equilibrium triangular finite element and the minimization of the complementary energy. This approach provides a reliable SA-field but its difficulty of implementation and its computational cost are prohibitive.

Figure~\ref{fig:carre_p} show the convergence rate of the estimators based on the data in Table~\ref{tab:carre_p}. 
\begin{figure}
\begin{minipage}{.44\textwidth}
\centering
\includegraphics[width=.5\textwidth]{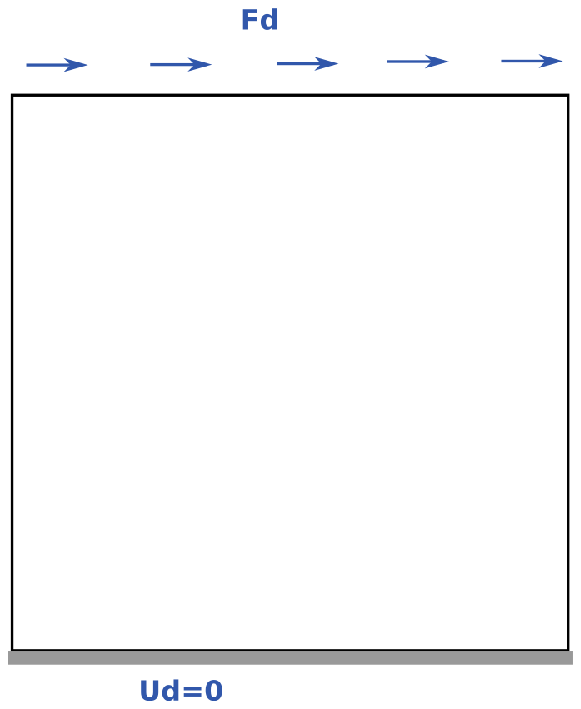}\caption{Square surface model problem}\label{fig:carre_pb}
\end{minipage}
\begin{minipage}{.55\textwidth}
\centering
\begin{tikzpicture}
\begin{loglogaxis}[
scale=.84,
xlabel=$h$,
legend style={at={(.04,1.2)}, anchor=north west}] 

\addplot[color=blue,mark=diamond*] table[x=h,y=EET] {carre_p.txt};
\addlegendentry{\begin{scriptsize}EET\end{scriptsize}}
\addplot table[x=h,y=STARFLEET] {carre_p.txt};
\addlegendentry{\begin{scriptsize}{\sfs $_{\lVert\cdot\rVert_2}$}\end{scriptsize}}
\addplot table[x=h,y=STARFLEET2] {carre_p.txt};
\addlegendentry{\begin{scriptsize}{\sfs $_{\lVert\cdot\rVert_{erdc}}$}\end{scriptsize}}
\addplot table[x=h,y=SPET] {carre_p.txt};
\addlegendentry{\begin{scriptsize}{Flux-free}\end{scriptsize}}
\addplot[draw=violet, mark=triangle*] table[x=h,y=dual] {carre_p.txt};
\addlegendentry{\begin{scriptsize}Dual approach\end{scriptsize}}
\addplot[color=black] coordinates {
(0.1, 0.09)(0.1, 0.18)(0.05 , 0.09)(0.1, 0.09)
};
\addlegendentry{\begin{scriptsize}h-slope\end{scriptsize}}
\end{loglogaxis}
\draw (-0.5,2.2) node[scale=1.,rotate=90]{Error};
\end{tikzpicture}\caption{Square surface, error estimator (p+2)}\label{fig:carre_p}
\end{minipage}
\end{figure}

\begin{table}[ht]\centering
\begin{tabular}{|c|c|c|c|c|c|c|c|}
\hline 
ndof & h &  EET & {\small \sfs$_{\lVert\cdot\rVert_2}$}  & {\small \sfs $_{\lVert\cdot\rVert_{erdc}}$} & Flux-free &Dual approach 
 \\ \hline 
60 & 0.2  &  1.5859&1.3865 & 1,3191 & 1.0143 & 0.7958 
\\ \hline 
928& 0.05 & 0.46949 & 0.41694 & 0,39775 & 0.32142 &0.2496 
\\ \hline 
5820 & 0.02 &0.21495 &0.19493 & 0.18481 & 0.15041 &0.1128 
\\ \hline 
\end{tabular} 
\caption{Square surface, error estimators (p+2 version)}\label{tab:carre_p}
\end{table}

We observe that the dual approach is the most accurate followed by the flux-free approach, the approaches based on the strong prolongation give slightly less accurate estimates. In that case the {\small \sfs$_{\lVert\cdot\rVert_2}$} criterion is very near to the best achievable ({\small \sfs $_{\lVert\cdot\rVert_{erdc}}$}).


\subsection{Pathological case : poor mesh quality}
Let us reuse previous example with a mesh of poor quality where some triangles are very thin (the ratio between the base and the height is equal to $\frac{1}{10}$) as illustrated on  Figure~\ref{fig:maillage_deforme}. The bad behavior of the \textit{a posteriori} error estimators for such meshes is a well-known phenomenon and initiated improvements as in \cite{Flo03bis} where the strong prolongation equation was weaken and partly replaced by energy minimization. 

We observe that the optimal SF criterion and the Flux-Free technique give best estimators, whereas the {\small \textsc{starfleet}$_{\lVert\cdot\rVert_{2}}$} gives large overestimation. These results tend to prove the importance of minimizing the energy when dealing with meshes of poor quality.
 
 \begin{figure}[ht]
\centering
\includegraphics[width=.3\textwidth]{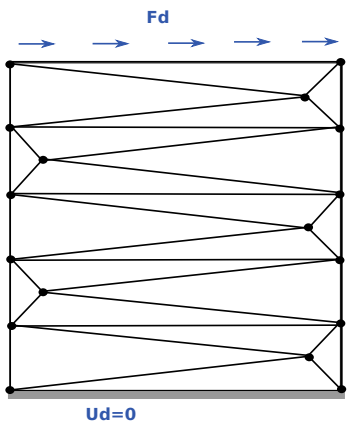}\caption{mesh with thin triangles}\label{fig:maillage_deforme}
 \end{figure}

\begin{table}[ht]\centering
 \begin{tabular}{|c|c|c|c|c|c|c|}
\hline 
ndof & h &  Flux-free &EET  & {\small \sfs$_{\lVert\cdot\rVert_2}$} & {\small \sfs$_{\lVert\cdot\rVert_{erdc}}$} & $\lVert \strain{\depH} \rVert_{ {\hooke}, \Omega}$ \\ 
\hline 
34 & 0.1 & 11.639 & 17.442 & 37.57 & 11.642 & 4.0425 \\ 
\hline 
\end{tabular} 
\caption{\label{tab:pathological_case} 2D square with poor quality mesh, error estimators (p+2 version)}
\end{table}

\subsection{Cracked structure}
Finally, let us consider the structure of Figure~\ref{fig:diez_pb} used in \citep{Par06}, which has two holes. We impose homogeneous Dirichlet boundary conditions in the central hole and on the base. The smaller hole is subjected to a constant unitary pressure $p_0$. A unitary traction force $\un{t}$ is applied normally to the surface on the left upper part. The other remaining boundaries are traction-free. A crack is also initiated from the smaller hole. 

Figure~\ref{fig:diez_p} and Table~\ref{tab:diez_p} show the convergence of the methods depending on the mesh size with p+2 refinement for the element computations. 

\begin{figure}[ht]
\begin{minipage}[b]{.45\textwidth}
\centering
\includegraphics[width=.6\textwidth]{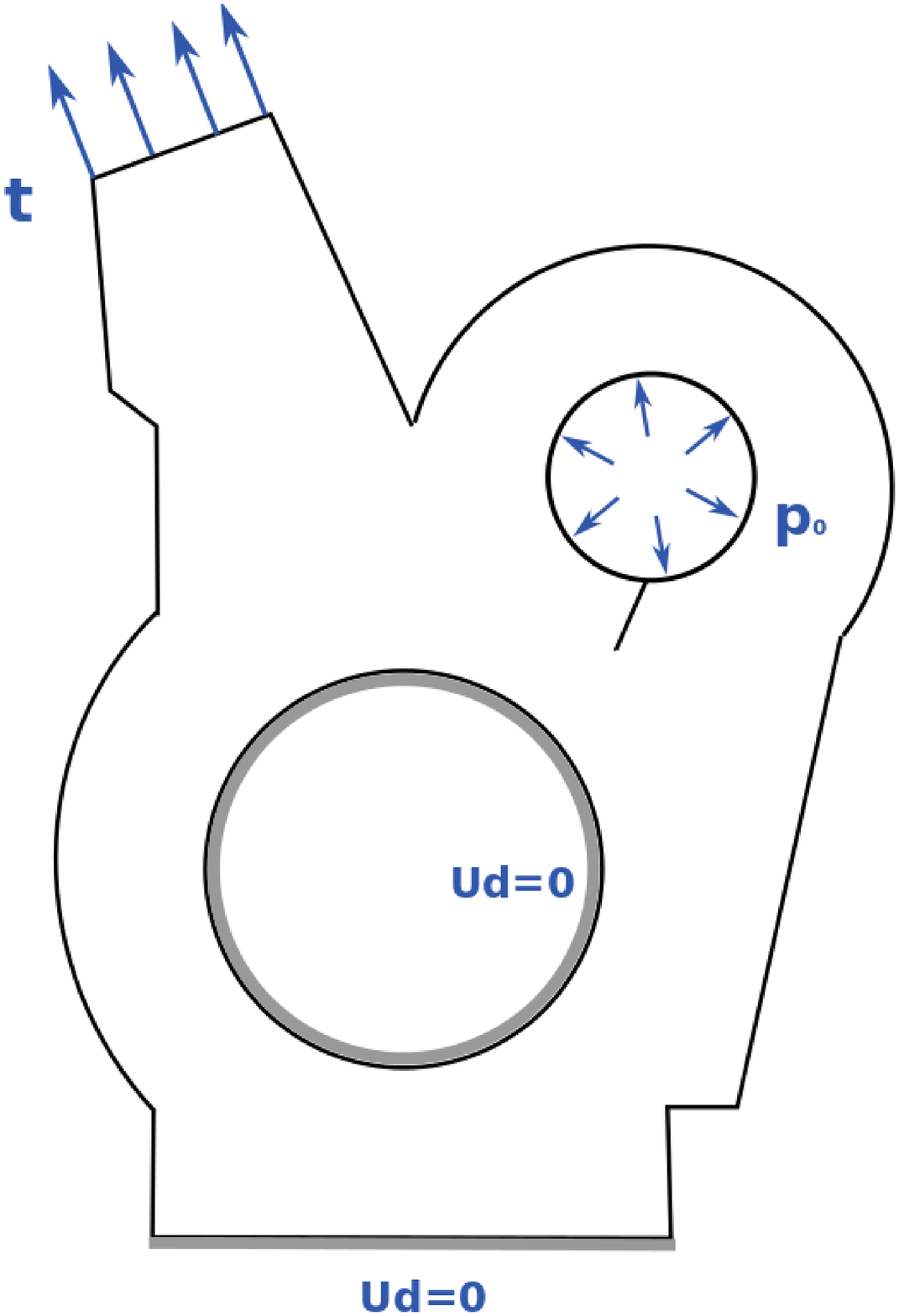}\caption{Crack-opening model problem}\label{fig:diez_pb}
\end{minipage}
\begin{minipage}[b]{.53\textwidth}
\centering
\begin{tikzpicture}
\begin{loglogaxis}[
scale=.83,
xlabel=$h$,
extra y ticks={0.14},
legend style={at={(.05,0.7)}, anchor=south west}] 
%
\addplot[color=blue,mark=diamond*] table[x=h,y=EET] {diez_p.txt};
\addlegendentry{\begin{scriptsize}EET\end{scriptsize}}
\addplot table[x=h,y=STARFLEET] {diez_p.txt};
\addlegendentry{\begin{scriptsize}{ \sfs $_{\lVert\cdot\rVert_2}$}\end{scriptsize}}
\addplot table[x=h,y=STARFLEET2] {diez_p.txt};
\addlegendentry{\begin{scriptsize}{ \sfs $_{\lVert\cdot\rVert_{erdc}}$}\end{scriptsize}}
\addplot table[x=h,y=SPET] {diez_p.txt};
\addlegendentry{\begin{scriptsize}{ Flux-free}\end{scriptsize}}
\addplot[color=black] coordinates {
(0.1, 0.2) (0.1, 0.3) (0.05 , 0.2) (0.1, 0.2) };
\addlegendentry{\begin{scriptsize}$h^\frac{1}{2}$-slope\end{scriptsize}}
\end{loglogaxis}
\draw (-0.5,2.1) node[scale=1.,rotate=90]{Error};
\end{tikzpicture}\caption{Crack-opening, estimators (p+2)}\label{fig:diez_p}
\end{minipage}
\end{figure}

\begin{table}[ht]
\centering
\begin{tabular}{|c|c|c|c|c|c|c|}
\hline 
ndof & h &  EET & {\small \sfs$_{\lVert\cdot\rVert_2}$}  & {\sfs$_{\lVert\cdot\rVert_{erdc}}$} & Flux-free& $\lVert \strain{\depH} \rVert_{ {\hooke}, \Omega}$ \\ 
\hline 
740 & 0.1 & 0.45098 & 0.53401 & 0.42509& 0.36538 &0.86727 \\ 
\hline 
1330& 0.07 & 0.34881 & 0.43558 & 0.32949& 0.28817 &0.88202 \\ 
\hline 
5638 & 0.03 & 0.22324 & 0.33203 & 0.17736 &0.1707  &0.89362\\ 
\hline 
12160& 0.02 & 0.17172 & 0.29363 & 0.13658 & 0.12724  &0.89648\\ 
\hline 
\end{tabular} 
\caption{Crack-opening, error estimators (p+2 version)}\label{tab:diez_p}
\end{table}

\begin{table}[ht]
\centering
\begin{tabular}{|c|c|c|c|c|c|}
\hline 
ndof & h &  EET & {\small \sfs$_{\lVert\cdot\rVert_{EET}}$} & {\sfs$_{\lVert\cdot\rVert_{erdc}}$}  &  Flux-free  \\ 
\hline 
740 & 0.1 &1 & 0.6 &1.4 &2.6 \\ 
\hline 
1330& 0.07 & 1 & 0.6 &1.2 & 2.1 \\ 
\hline 
5638 & 0.03 & 1 &0.7 & 3.5 & 2.4 \\ 
\hline 
12160& 0.02 & 1 & 0.9& >10 & 2.1\\ 
\hline 
\end{tabular} 
\caption{Crack-opening, normalized CPU times }\label{tab:diez_cpu}
\end{table}

The new SA-field technique  {\small \textsc{starfleet}$_{\lVert\cdot\rVert_2}$} behaves poorly when the mesh is refined. This may be due to the singularity of the problem. As expected, with the global minimization of the energy, the optimized strategy {\small \textsc{starfleet}$_{\lVert\cdot\rVert_{erdc}}$}  leads to better results than classical EET in particular for fine meshes, it tends to equal the Flux-Free method which is always the most accurate.

In Table~\ref{tab:diez_cpu} we present sequential CPU-times on a workstation potentially shared with other users, which makes the measures only indicative. We observe that the \textsc{starfleet} implementation of the EET is always the fastest, that the optimized \textsc{starfleet} is not scalable, and that the Flux-Free technique is always more than 2 times slower than EET.

%% file: extensions.tex
In this section we discuss the possibility to use the method on more complex cases. {We only discuss the construction of the system related to the strong prolongation hypothesis equivalent to~\eqref{eq:todo1}. Since components are decoupled, we can derive our analysis on a scalar problem.}

\subsection{Higher order triangular elements}
Let us consider triangular straight Lagrange elements of order 2 (obtained by a linear transformation from the reference triangle). For these elements, some nodes are not vertexes. We write the strong prolongation equation with the test fields $(1,x,y, x^2, xy, y^2 )$. Second order works have to satisfy two conditions to be consistent with the geometry:
\begin{equation}
 a_\Gamma x + b_\Gamma y + c_\Gamma  = 0 \Longrightarrow\ \left\{\begin{aligned}
 a_\Gamma W^\Gamma(x^2) + b_\Gamma  W^\Gamma(yx) + c_\Gamma  W^\Gamma(x) &= 0\\
 a_\Gamma W^\Gamma(xy) + b_\Gamma  W^\Gamma(y^2) + c_\Gamma  W^\Gamma(y) &= 0\\
\end{aligned}\right.
\end{equation}
 The final system then writes:
\begin{equation}
\begin{pmatrix}
\matD & & & & & \\
& \matD & & & &  \\
& & \matD & & &  \\
& & & \matD & &  \\
& & & & \matD &  \\
& & & & & \matD  \\
\matc &  \mata & \matb & & &\\
& \matc &  & \mata & &  \matb \\
& & \matc &  & \matb & \mata 
\end{pmatrix}
\begin{pmatrix}
\vecW(1) \\\vecW( x) \\\vecW(y) \\\vecW(x^2) \\\vecW(y^2) \\\vecW(xy)
\end{pmatrix}=
\begin{pmatrix}\vecR(1) \\ \vecR(x)\\ \vecR(y)\\ \vecR(x^2) \\  \vecR(y^2) \\ \vecR(xy) \\0 \\ 0 \\ 0
\end{pmatrix}
\end{equation}

The kernel of this new matrix $\matG$ is still composed of the vectors describing the internal star-patches. Then
 \begin{equation}
\operatorname{dim}(\ker(\matG))= \setvertexi
\end{equation}
Therefore the dimension of the left kernel of matrix $\matG$ is
\begin{equation}
\operatorname{dim}(\ker(\matG^T))=6\card{\setelem}+3\card{\setedgei}- (6\card{\setedgei}-\card{\setvertexi})
\end{equation}
Each element being made out of 3 edges, each internal edge belonging to 2 elements and each external edge belonging to one element, we have $3\card{\setelem} -(2\card{\setedgei} +\card{\setedgee})=0$.
In the end 
\begin{equation}
\operatorname{dim}(\ker(\matG^T))=\card{\setedgei} +2\card{\setedgee}+\card{\setvertexi}= \card{\setedge}+ \card{\setvertex}
\end{equation}
which is exactly the number of nodes of the discrete problem. The interpretation is similar to the case of linear triangular elements.

Once the works are known, one possibility for the definition of the higher order boundary conditions is to assume that they are associated with qua\-dra\-tic distributions of traction force. The last step of the procedure (solving local problems with Neumann boundary conditions) is unmodified.

\begin{remark} 
Isoparametric elements can have curved edges which makes the compatibility condition more complex to write. Moreover the idea to seek distribution of tractions of given shape (linear, quadratic\ldots) is questionable, which makes the definition of higher order problems more difficult. We leave this question open for future investigation.
\end{remark} 

\subsection{Quadrangular elements}
With quadrangular elements, the finite element basis of shape functions is richer. Therefore, we write the strong prolongation equation with the test fields $(1,x,y, xy)$. The system writes :

\begin{equation}
\begin{pmatrix}
\matD & & &  \\
& \matD & &  \\
& & \matD &   \\
& & & \matD \\
\matc &  \mata & \matb & \\
0 & \tilde{\mathbf{y}}_0 & \tilde{\mathbf{x}}_0  & - \tilde{\matI}
\end{pmatrix}
\begin{pmatrix}
\vecW(1) \\\vecW( x) \\\vecW(y) \\\vecW(xy) 
\end{pmatrix}=
\begin{pmatrix}\vecR(1) \\ \vecR(x)\\ \vecR(y)\\ \vecR(xy)  \\ 0 
\end{pmatrix}
\end{equation}
In general, the work $\vecW(xy)$ is linearly independent from the first order works $\vecW(1)$ ,$\vecW( x)$ and $\vecW( y)$. The vertical and horizontal edges  constitute the exception: indeed if the equation of the edge is $x=\tilde{x}_0$ (vertical edge) then $W(xy)=x_0 W(y)$. The diagonal matrices $\tilde{\mathbf{x}}_0$ and $\tilde{\mathbf{y}}_0$ denote these specific edges.

One solution to get rid of this difficulty is to \textit{a priori} seek works associated to linear distribution of tractions. This is the strategy retained in the classical EET technique, which warranties the existence of a solution \citep{Lad01}. In that case, the first order works are sufficient to determine the distribution and $\vecW(xy)$ is simply post-processed in order to define the loading for the element computations.

We end up with the same system as for triangular elements:
\begin{equation}
\begin{pmatrix}
\matD & &   \\
& \matD &   \\
& & \matD    \\
\matc &  \mata & \matb  
\end{pmatrix}
\begin{pmatrix}
\vecW(1) \\\vecW( x) \\\vecW(y)  
\end{pmatrix}=
\begin{pmatrix}\vecR(1) \\ \vecR(x)\\ \vecR(y)\\ 0 
\end{pmatrix}
\end{equation}
The kernel of the matrix $\matG$ is still composed of the vectors describing the internal star-patches. Then
 \begin{equation}
\operatorname{dim}(\ker(\matG))= \setvertexi
\end{equation}
Therefore the dimension of the left kernel of matrix $\matG$ is
\begin{equation}
\operatorname{dim}(\ker(\matG^T))=4\card{\setelem} +2\card{\setedgei}-(4\card{\setedgei} -\card{\setvertexi})
\end{equation}
Each element being made out of 4 edges, each internal edge belonging to 2 elements and each external edge belonging to one element, we have $4\card{\setelem} -(2\card{\setedgei}+\card{\setedgee})=0$.
In the end 
\begin{equation}
\operatorname{dim}(\ker(\matG^T))= \card{\setvertex}
\end{equation}
which correspond to the FE equilibrium for each shape function.

\subsection{3D problems}
We consider a 3D tetrahedral mesh. We introduce the set of faces $\setface$. 
The unknowns are the works on internal faces. We write the prolongation equation with the following test fields:
\begin{equation*}
\left(1,x,y, z \right)
\end{equation*}

Topological Matrix $\matD$ is a rectangular matrix with $\card{\setelem}$ rows and $\card{\setfacei}$ columns. Orientation of tetrahedron being non intuitive, a simple way to fill the matrix is to give an arbitrary orientation to each face, which corresponds to the definition of a normal vector, and to count positively the tetrahedron which is pointed inward by the normal and negatively the one pointed outward.
Then $\matD$ has still only one vector in its left kernel: $\veci=\begin{pmatrix} 1&\ldots&1\end{pmatrix}^T$.

In order to analyze the right kernel, we need the Euler-Poincaré relationship for one 3D connected component and its skin. The analysis is more complex because we need to distinguish between the $h$ circular holes (like donuts)  and the $p$ 3D voids (like  the stone of a peach), in algebraic topology $h$ and $p$ are particular instances of Betti numbers \citep{godbillon98}. We have the following properties:
\begin{equation}\label{eq:3DEuler}
\begin{aligned}
\card{\setvertex} - \card{\setedge} + \card{\setface}  - \card{\setelem}&=1-h+p\\
\card{\setvertexe} - \card{\setedgee} + \card{\setfacee} &= 2(1-h+p)
\end{aligned}
\end{equation}

Moreover, we can add the following simple relationships obtained by counting faces of tetrahedrons and edges of the skin:
\begin{equation}\label{eq:3Dcounting}
\begin{aligned}
4\card{\setelem}&=2 \card{\setfacei}+ \card{\setfacee} \\
3 \card{\setfacee} &= 2\card{\setedgee}
\end{aligned}
\end{equation}

We deduce the dimension of the right kernel of $\matD$ :
\begin{equation}
\begin{aligned}
\operatorname{dim}(\ker(\matD))&=\card{\setfacei}- \card{\setelem} +\operatorname{dim}(\ker(\matD^T))\\
&= \card{\setfacei} +1-h+p-\card{\setface}+\card{\setedge} -\card{\setvertex}+1\\
&=\card{\setvertexe} -\card{\setedgee} +2-h+p-\card{\setedge} -\card{\setvertex}-2(1+p-h)\\
&=\card{\setedgei} -\card{\setvertexi}+h-p
\end{aligned}
\end{equation}
This result has a simple interpretation: internal edges define close circuit between faces (like internal nodes in 2D), but internal nodes are associated to one redundancy (see Figure~\ref{fig:cube}). Each hole offers one extra independent close circuit whereas each pore imposes one extra redundancy.
\begin{figure}
\centering
\includegraphics[width=.4\textwidth]{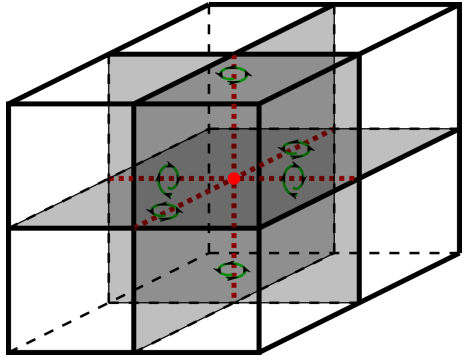}\caption{Simple 3D case, internal faces, edges and node}\label{fig:cube}
\end{figure}


As in 2D problems, works on faces are not independent. For a given face $F$, there exists geometrical real coefficients $(a_F,b_F,c_F,d_F)$ such that,
\begin{equation}
 (x,y,z) \in F \Longleftrightarrow\ a_F x + b_F y + c_F z + d_F = 0
\end{equation}
the coefficients $a_F$, $b_F$ and $c_F$ can be chosen so that $a_F^2+b_F^2 + c_F^2 =|F|^2$ (area of the face). 
Therefore $a_F W^F(x)+b_F W^F(y)+c_F W^F(z)+ d_F W^F(1)=W^F(a_F x +b_F y +c_F z + d_F )=0$.

The global system writes:
\begin{equation}
\begin{pmatrix}
\matD & & &  \\
& \matD & &  \\
& & \matD &   \\
& & & \matD \\
\matd &  \mata & \matb & \matc &
\end{pmatrix}
\begin{pmatrix}
\vecW(1) \\\vecW( x) \\\vecW(y) \\\vecW(z) 
\end{pmatrix}=
\begin{pmatrix}\vecR(1) \\ \vecR(x)\\ \vecR(y)\\ \vecR(z)  \\ 0 
\end{pmatrix}
\end{equation}
Let $\matG$ be the large sparse matrix in previous equation. Let $N_e$ be a kernel vector of $\Delta$ associated to edge $e$, and $(x^e_1,y^e_1,z^e_1)$ and $(x^e_2,y^e_2,z^e_2)$ be two distinct points of $e$, we have:
\begin{equation}
\matG \begin{pmatrix}  N_e & N_e \\ x^e_1 N_e & x^e_2 N_e \\ y^e_1N_e & y^e_2 N_e \\ z^e_1N_e &z^e_2 N_e  \end{pmatrix} = 0
\end{equation}
Then each internal edge generates two kernel vectors. Here again each internal node generates one redundancy in the kernel. So that
\begin{equation}
\operatorname{dim}(\operatorname{ker}(\matG))= 2\card{\setedgei} -\card{\setvertexi}
\end{equation}

Using successively the rank theorem, formulas \eqref{eq:3Dcounting} and \eqref{eq:3DEuler}, we find that: 
\begin{equation}
\begin{aligned}
\operatorname{dim}(\operatorname{ker}(\matG^T))&=4\card{\setelem}  +\card{\setfacei} -(4\card{\setfacei} -\operatorname{dim}(\operatorname{ker}(\matG))) \\
&=4\card{\setelem}  -3 \card{\setfacei} +2 \card{\setedgei} -\card{\setvertexi} = -\card{\setfacei} + \card{\setfacee} +2 \card{\setedgei} -\card{\setvertexi} \\
&=2(1-h+p)+\card{\setedgee}-\card{\setvertexe} -\card{\setfacei} +2 \card{\setedgei} -\card{\setvertexi} \\
&=\card{\setvertex} +(-\card{\setedgee} +2 \card{\setface} -\card{\setfacei} -2\card{\setelem} )\\
&=\card{\setvertex} +(-\frac{3}{2}\card{\setedgee} +2 \card{\setface} -2 \card{\setfacei} -\frac{1}{2}\card{\setfacee} )=\card{\setvertex}
\end{aligned}
\end{equation}
which corresponds to the finite element equilibrium for each node.

The rest of the resolution can be conducted the same way as for 2D problems.

%% file: conclusion.tex
This paper introduces a new approach for the construction of statically admissible stress fields for \textit{a posteriori} error estimation. More precisely, the method builds a full basis of the edge (in 2D) traction fields satisfying the strong prolongation equation, which enables to try various optimization criteria. In particular, through the choice of the criterion, it is possible to recover the classical EET technique or to minimize the error estimation.

From a practical point of view, the method relies on the exploitation of the known quantities on the boundary of the domain and on a topological and geometrical analysis of the mesh at the global scale. The associated linear systems are so sparse that their size is not problematic, moreover the full knowledge of the kernels (for any order, shape or dimensionality) makes it possible to extract well posed problems where fast solvers can be employed.

The presentation and assessments are given for 2D static linear mechanics problems meshed with linear triangular elements, but extensions to elements with higher order, different shape and different dimensionality are described. In particular, the treatment of 3D problems (with their known kernels) seems no more difficult than 2D problems, which is a feature not encountered in many other SA-field recovery techniques.

{We observe that estimators based on the strong prolongation equation are always less precise than the flux-free technique (and of course dual approaches). The minimization of the estimator inside SA-fields satisfying the strong prolongation equation tends to give results very close to the flux free technique but at a much higher cost. The classical EET criterion leads to a faster estimator, especially with its new implementation, and its quality is not far from the optimized estimator. }

This means that the strong prolongation equation is {a questionable hypothesis} for people willing to build very precise estimators. Note that this lack of pertinence was known since a weak prolongation equation was proposed in \citep{Flo03bis}  to deal with very deformed elements and in \citep{Ple12} to improve the efficiency of the estimator in zones where the error is concentrated. On the other hand, the strong prolongation equation leads to  lower computational costs for the error estimation. {Criteria which barely perturb the sparsity of the problem (like the norm-2 criterion or the classical EET criterion) lead to lower CPU times than classical implementation of EET (which itself is faster than the flux-free approach). Beside the implementation of the our method is simpler and more easily extensible to different elements. Moreover, we expect the method to give interesting performance in cases where matrices can be reused (like the error estimation of a sequence of problems based on the same mesh).}